\documentclass[12pt,a4paper,twoside]{article}

\usepackage{amsmath , amsfonts, amssymb, amsthm}
\usepackage[all]{xy}
\usepackage[a4paper,top=3cm,bottom=3cm,left=2.5cm,right=2.5cm]{geometry}
\usepackage{tikz}

\newtheorem{theorem}{Theorem}
\newtheorem{corollary}[theorem]{Corollary}
\newtheorem{lemma}[theorem]{Lemma}
\newtheorem{proposition}[theorem]{Proposition}
\newtheorem{conjecture}[theorem]{Conjecture}
\newtheorem*{theoremz}{Theorem}

\theoremstyle{definition}
\newtheorem{definition}[theorem]{Definition}
\newtheorem{example}[theorem]{Example}
\newtheorem{remark}[theorem]{Remark}

\title{The regularity of almost all edge ideals}
\author{Alexander Engstr\"om\footnote{\tt alexander.engstrom@aalto.fi}  \and Milo Orlich\footnote{\tt milo.orlich@aalto.fi}}

\def\til{~}

\def\reg{\mathrm{reg}}

\begin{document}

\maketitle

\begin{abstract}
A fruitful contemporary paradigm in graph theory is that almost all graphs that do not contain a certain subgraph have common structural characteristics. An example is the well-known result saying that almost all triangle-free graphs are bipartite. The ``almost'' is crucial, without it such theorems do not hold. In this paper we transfer this paradigm to commutative algebra and make use of deep graph theoretic results. A key tool are the \emph{critical graphs} introduced relatively recently by Balogh and Butterfield, who proved that almost all graphs not containing a critical subgraph have common structural characteristics analogous to being bipartite.

For a graph $G$, let $I_G$ denote its edge ideal, the monomial ideal generated by $x_ix_j$ for every edge $ij$ of $G$. In this paper we study the graded Betti numbers of $I_G$, which are combinatorial invariants that measure the complexity of a minimal free resolution of $I_G$. The Betti numbers of the form $\beta_{i,2i+2}$ constitute the ``main diagonal'' of the Betti table. It is well known that for edge ideals any Betti number  to the left of this diagonal is always zero. We identify a certain ``parabola'' inside the Betti table and call \emph{parabolic Betti numbers} the entries of the Betti table bounded on the left by the main diagonal and on the right by this parabola. Let $\beta_{i,j}$ be a parabolic Betti number on the $r$-th row of the Betti table, for $r\ge3$. We prove that almost all graphs $G$ with $\beta_{i,j}(I_G)=0$ can be partitioned into $r-2$ cliques and one independent set. In particular, for almost all graphs $G$ with $\beta_{i,j}(I_G)=0$, the regularity of $I_G$ is $r-1$. 
\end{abstract}


\section{Introduction}

In this paper we consider edge ideals $I_G$ and their Betti numbers $\beta_{i,j}(I_G)$. It is well known that these numerical invariants of $I_G$ can be computed with Hochster's formula by studying the simplicial homology of the independence complex $\mathrm{Ind}(G)$ of $G$ and its subcomplexes. Here $\mathrm{Ind}(G)$ is the simplicial complex whose faces are the independent sets of $G$. Ultimately, this amounts to understanding whether $G$ contains some specific induced subgraphs or not, and this is a problem that can be naturally phrased in terms of extremal graph theory.

For graphs $G$ and $H$, one says that $G$ is \emph{$H$-free} if $G$ does not contain a copy of $H$ as an induced subgraph. 
Balogh and Butterfield\til\cite{BaBu} introduced relatively recently the concept of a \emph{critical graph}, a rather technical notion, giving a characterization for critical graphs $H$ in terms of almost all graphs that are $H$-free. This is the key tool for our main results, and we believe that we are the first to employ this method in commutative algebra.

Recall that the Betti numbers on the $r$-th row of the Betti table are those of the form $\beta_{i,i+r}$, for $i\ge0$.
In this paper we say that a Betti number $\beta_{i,i+r}$, for $r\ge3$, is  a \emph{parabolic Betti number} if
$$r-2\le i\le r-2+\binom{r-1}2.$$
The region of the Betti table consisting of the parabolic Betti numbers is bounded from the left by the diagonal consisting of $\beta_{1,4},\beta_{2,6}, \beta_{3,8},\dots$, and from the right by the ``parabola'' consisting of $\beta_{1,4},\beta_{2,6},\beta_{3,7},\dots$ In Figure\til\ref{fig:starredregion}
\begin{figure}
\begin{center}
\begin{tikzpicture}[>=latex,scale=0.4]
\draw [help lines, very thin] (0,-9.75) grid (37.75,0);
\draw[rotate around={270:(0,0)},thick] (0,.375) parabola (8.64581,37.75);
\draw[rotate around={270:(0,0)},very thick,white] (0,.375) parabola (1.5,1.5);
\fill[white] (.1,-.4) rectangle (.9,-.03);
\fill (.5,-.5) circle (.1);
\fill (1.5,-1.5) circle (.1);
\fill (3.5,-2.5) circle (.1);
\fill (6.5,-3.5) circle (.1);
\fill (10.5,-4.5) circle (.1);
\fill (15.5,-5.5) circle (.1);
\fill (21.5,-6.5) circle (.1);
\fill (28.5,-7.5) circle (.1);
\fill (36.5,-8.5) circle (.1);
\draw [help lines, very thin] (0.8,-1.2) grid (1.1,-.8);
\draw [very thick] (0,1) -- (0,-9.75);
\draw [very thick] (-1,0)-- (37.75,0);
\fill (2.5,-2.5) circle (.1);
\fill (3.5,-3.5) circle (.1);
\fill (4.5,-4.5) circle (.1);
\fill (5.5,-5.5) circle (.1);
\fill (6.5,-6.5) circle (.1);
\fill (7.5,-7.5) circle (.1);
\fill (8.5,-8.5) circle (.1);
\fill (9.5,-9.5) circle (.1);
\draw[thick] (.5,-.5)--(9.75,-9.75);
\fill [gray, opacity=0.3] (1,-2) rectangle (2,-1);
\fill [gray, opacity=0.3] (2,-3) rectangle (4,-2);
\fill [gray, opacity=0.3] (3,-4) rectangle (7,-3);
\fill [gray, opacity=0.3] (4,-5) rectangle (11,-4);
\fill [gray, opacity=0.3] (5,-6) rectangle (16,-5);
\fill [gray, opacity=0.3] (6,-7) rectangle (22,-6);
\fill [gray, opacity=0.3] (7,-8) rectangle (29,-7);
\fill [gray, opacity=0.3] (8,-9) rectangle (37,-8);
\fill [gray, opacity=0.3] (9,-9.75) rectangle (37.75,-9);
\node [font=\footnotesize] at (.5,.5) {0};
\node [font=\footnotesize] at (1.5,.5) {1};
\node [font=\footnotesize] at (3.5,.5) {3};
\node [font=\footnotesize] at (6.5,.5) {6};
\node [font=\footnotesize] at (10.5,.5) {10};
\node [font=\footnotesize] at (15.5,.5) {15};
\node [font=\footnotesize] at (21.5,.5) {21};
\node [font=\footnotesize] at (28.5,.5) {28};
\node [font=\footnotesize] at (36.5,.5) {36};
\node [font=\footnotesize] at (-.5,-.55) {2};
\node [font=\footnotesize] at (-.5,-1.55) {3};
\node [font=\footnotesize] at (-.5,-2.55) {4};
\node [font=\footnotesize] at (-.5,-3.55) {5};
\node [font=\footnotesize] at (-.5,-4.55) {6};
\node [font=\footnotesize] at (-.5,-5.55) {7};
\node [font=\footnotesize] at (-.5,-6.55) {8};
\node [font=\footnotesize] at (-.5,-7.55) {9};
\node [font=\footnotesize] at (-.58,-8.55) {10};
\end{tikzpicture}
\caption{The parabolic Betti numbers in the first few rows of the Betti table are marked by gray squares.}
\label{fig:starredregion}
\end{center}
\end{figure}
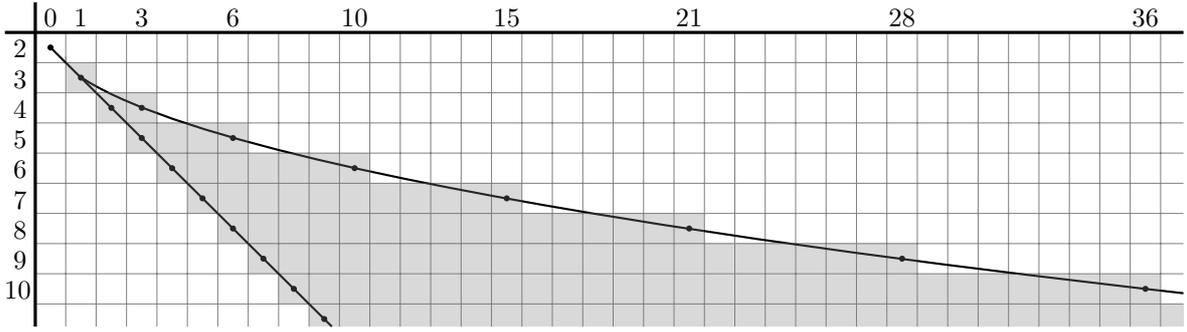
the Betti numbers on row $r$, for $3\le r\le 10$, are marked with gray squares, and  the diagonal and the parabola bounding the region of parabolic Betti numbers are also drawn. The reason for considering these parabolic Betti numbers $\beta_{i,j}(I_G)$ is that when one of them is zero, we will show that $G$ is $H$-free for some critical graph $H$.

In the main results below, when we say that \textit{almost every graph} in some family 
$\mathcal A$ is in some family $\mathcal B$, we mean the following: for every $n$, if 
$\mathcal A_n$ and $\mathcal B_n$ denote respectively the subsets of $\mathcal A$ and $\mathcal B$ whose elements have $n$ vertices, we have $\mathcal B_n\subseteq\mathcal A_n$ and $\lim_{n\to\infty}\frac{|\mathcal A_n|}{|\mathcal B_n|}=1$.
 We say that a graph $G$ is an \emph{$(s,t)$-template} if $G$ can be covered with $s$ cliques (i.e., complete graphs) and $t$ independent sets. One of our main results describes the structure of almost all graphs with some vanishing parabolic Betti number:

\begin{theoremz}[Theorem\til\ref{thm:main1}]
Let\/ $\beta_{i,j}$ be a parabolic Betti number on the $r$-th row of the Betti table, for some $r\ge3$. Then almost every graph $G$ with $\beta_{i,j}(I_G)=0$ is an $(r-2,1)$-template.
\end{theoremz}

Recall that the \emph{regularity} of $I_G$, denoted $\reg(I_G)$, is the largest index of a non-zero row in the Betti table. The other main result, which follows from the theorem above and an additional argument, is the following:

\begin{theoremz}[Theorem\til\ref{thm:main2}]
Let\/ $\beta_{i,j}$ be a parabolic Betti number on the $r$-th row of the Betti table, for some $r \geq 3$. Then, for almost every graph $G$ with $\beta_{i,j}(I_G)=0$, we have:
\begin{enumerate}
\item $\reg(I_G)=r-1$, and
\item every parabolic Betti number of\/ $I_G$ above row $r$ is non-zero.
\end{enumerate}
\end{theoremz}

\subsection{Outline of the paper}

In Section\til\ref{sec:critical graphs} we recall the necessary graph theory background, with various examples so as to give a friendly introduction to critical graphs. In Section\til\ref{sec:backalgebra} we recall some notions in commutative algebra, in particular Betti numbers and Hochster's formula.

In Section\til\ref{sec:parabolic} we define the main objects that come into play: parabolic clusters and parabolic Betti numbers. The parabolic clusters are the critical graphs to which we apply the machinery of Balogh and Butterfield\til\cite{BaBu}, and they are the graphs that do not appear as induced subgraphs of $G$ exactly when some suitable parabolic Betti number $\beta_{i,j}(I_G)$ is zero.

In Section\til\ref{sec:d1templates} we prove some preliminary results concerning $(d,1)$-templates, that is, graphs that can be covered with $d$\til{cliques} and one independent set. In Section\til\ref{sec:main} we state and prove the two main theorems mentioned above.

One may wonder how ``sharp'' our results are.  More precisely, we prove that for \emph{almost} all graphs  with a parabolic Betti number $\beta_{i,j}(I_G)=0$ on the $r$-th row, we get $\reg(I_G)=r-1$. This (as recalled above) means that when we take the graphs on $n$ vertices satisfying the repsective properties, the ratio of cardinalities is asymptotically equal to $1$. One may ask if the ``almost'' is just due to the special methods we employed, and how often it actually happens that the regularity is \emph{not} equal to $r-1$: we answer this question in Section\til\ref{sec:notrminus1}. The other natural question is whether one may be able to generalize the results above to arbitrary Betti numbers, and not just the parabolic ones: in Section\til\ref{sec:non-parabolic} we show how this  falls apart in the case of non-parabolic Betti numbers. Related to this, the paper~\cite{CoKaVa} (in particular Corollary 6.12) addresses the problem of producing squarefree monomial ideals that have arbitrarily high regularity despite having linear syzygies in the first arbitrarily many steps.

In Section\til\ref{sec:erdos} we relate our results to the famous Erd\H{o}s--Hajnal conjecture. Lastly, in Section\til\ref{sec:spaceofgraphs} we consider a tentative  ``space of graphs'' and prove that the $(d,1)$-templates are connected in that space, and in Section\til\ref{sec:future} we suggest possible future directions.

\subsection{Relation to papers on edge ideals of random graphs}

In this paper we work with \emph{unlabeled} graphs. Most papers on resolutions of ideals of random graphs and other monomial ideals consider labeled graphs instead, so that one would for instance make a distinction between the ideals $(xy,yz)$ and $(xz,zy)$ in $k[x,y,z]$. With a slight abuse of notation, we consider \emph{edge ideals on unlabeled graphs}, essentially meaning that we do not distinguish $(xy,yz)$ from $(xz,zy)$: indeed, we are only interested in the Betti numbers of edge ideals, and those do not depend on the labeling of the vertices.
There are geometric and probabilistic pros and cons with both approaches. In selecting almost all graphs, we do not privilege those with a particular number of edges. By keeping the number of edges low instead of requiring a parabolic Betti number to be zero, Erman and Yang\til\cite{EY} and Dochtermann and Newman\til\cite{DN} proved similar regularity and vanishing results as in our main theorems. Both of those papers build on the theory of random flag complexes developed by Kahle\til\cite{K1,K2,K3}.

There are more recent interesting results on resolutions of ideals of random graphs, see for example\til\cite{BW,BY,BEY,LHKS,LPSSW,EGS,SWY}.

\subsubsection*{Acknowledgements}

We thank two anonymous referees and Russ Woodroofe for many helpful comments, which greatly helped to improve the paper.
The first author would like to thank Magnus Halldorsson, Klas Markstr\"om, Andrzej Ruci\'nski, Carsten Thomassen and Institut Mittag-Leffler for organising the research program \textit{Graphs, Hypergraphs and Computing}.
The second author was supported by the Finnish Academy of Science and Letters, with the \textit{Vilho, Yrj\"o and Kalle V\"ais\"al\"a Fund}.

\section{Background}

\subsection{Tools from graph theory: critical graphs}\label{sec:critical graphs}

In this section we recall all the necessary notions from graph theory. The definition of critical graph (see Definition\til\ref{def:critical}) is quite technical and requires several preliminary concepts, in particular that of an ``$(s,t)$-template'', which is a term we coin. For the rest we  follow~\cite{BaBu} closely, both in notation and terminology. Because the concept of critical graph is not yet standard in commutative algebra, we offer several examples meant for the more algebraically-inclined readers to illustrate all the unavoidable notational technicalities.

A \emph{finite simple graph} is a pair $(V_G,E_G)$ where $V_G$ is the (finite) set of \emph{vertices} of $G$ and $E_G\subseteq\binom{V_G}2$ is the set of \emph{edges}. That is, the edges have no direction and we allow no multiple edges nor loops.
For a subset $U\subseteq V_G$, we denote by $G[U]$ the \emph{subgraph of $G$ induced by $U$}, which by definition is the subgraph with vertex set $U$ and the edges of\til$G$ with both endpoints in $U$.

By a \emph{graph} we will actually most often mean an \textbf{unlabeled} graph, i.e., the isomorphism class of a finite simple graph. One may still consider the edge ideal of an unlabeled graph, up to permutation of the variables of the polynomial ring. We observe that this does not affect the results of this paper.

Recall that a \emph{clique} (or \emph{complete graph}) is a graph where  any two vertices are adjacent, namely  connected by an edge. We denote by $K_n$ the clique on $n$ vertices. An \emph{independent set} is the complement of a clique. 

A coloring of a graph $G$ is a partition of the vertex set $V_G$ into independent sets. More formally, a $k$-\emph{coloring} of a graph\til$G$ is  a function $f\colon V_G\to[k]=\{1,\dots,k\}$ such that, if we denote $V_i:=\{v\in V_G\mid f(v)=i\},$ then each induced subgraph $G[V_i]$ is an independent set. More explicitly, in this situation one doesn't allow adjacent vertices to have the same ``color'' (i.e., the value of $f$). If such a coloring exists, one says that $G$ is \emph{$k$-colorable}. 

In this paper we consider a more general version of graph coloring, which is also well studied in graph theory but perhaps less popular in other branches of math. We partition the vertex set $V_G$ not just into independent sets, but instead into cliques and independent sets:

\begin{definition}\label{defcovtemp}
We say that a pair of non-negative integers $(s,t)$ is a \emph{covering pair} for a graph $G$ if there is a function $f\colon V_G\to[s+t]$ such that, if we set $V_i:=\{v\in V_G\mid f(v)=i\}$, then $G[V_i]$ is a clique for $1\le i\le s$ and $G[V_i]$ is an independent set for $s+1\le i\le s+t$. If $(s,t)$ is a covering pair for $G$, we call $G$ an \emph{$(s,t)$-template}.
\end{definition}

One may observe the following:
\begin{enumerate}
\item Any of the cliques or independent sets in the definition above may be empty. So in particular if a graph $G$ is an $(s,t)$-template, then $G$ is also an $(s',t')$-template for any $s'\ge s$ and $t'\ge t$.
\item The $(0,2)$-templates are the bipartite graphs. The $(0,k)$-templates are exactly the $k$-colorable graphs in the ``classical'' sense, and the $(1,1)$-templates are commonly known as split graphs.
\end{enumerate}

\begin{example}
Consider $P_5$, the path on five vertices. The pair $(2,0)$ is not a covering pair for $P_5$ because the largest cliques in $P_5$ are edges, and two edges are not enough to cover all of $P_5$. On the other hand $(2,1)$ is a covering pair for $P_5$: one may pick as independent set the middle point of $P_5$ and as the two cliques the first and last edge. Notice that this is not the unique way to cover $P_5$ with two cliques and one independent set. For instance, the independent set could also consist of the first, third and fifth vertex, and the two cliques would then consist of the two remaining vertices, one vertex each.
\end{example}

\begin{definition}
The \emph{coloring number} of a graph~$G$, denoted $\chi_{\mathrm c}(G)$, is the minimum $k$ such that every pair $(s,t)$ of non-negative integers with $s+t=k$ is a covering pair for~$G$. If a pair $(s,t)$ with $s+t=\chi_{\mathrm c}(G)-1$ is not a covering pair for~$G$, we call $(s,t)$ a \emph{witnessing pair} for $G$.
\end{definition}

The concept of coloring number was introduced in\til\cite{AleRussian,BoTh,PSdue}. It is also known in the literature as ``binary chromatic number'', for instance in \cite{AKMedit,BaMaedit}.

\begin{example}\label{examplecoloring}
Consider the five-cycle $C_5$ and the seven-cycle $C_7$ (drawn in Figure~\ref{fig:cyclescoloring} to help visualize things). 
\begin{figure}
\begin{center}
\begin{tikzpicture}[dot/.style={draw,fill,circle,inner sep=1pt}]
  \foreach \l [count=\n] in {0,1,2,3,4} {
    \pgfmathsetmacro\angle{90-360/5*(\n-1)}
      \node[dot] (n\n) at (\angle:1) {};
    \fill (\angle:1)  circle (0.1);
  }
  \draw[thick] (n1) -- (n2) -- (n3) -- (n4) -- (n5)  -- (n1);
  \coordinate [label=$C_5$] (G) at (0,-1.7);
\end{tikzpicture}\qquad\qquad\begin{tikzpicture}[dot/.style={draw,fill,circle,inner sep=1pt}]
  \foreach \l [count=\n] in {0,1,2,3,4,5,6} {
    \pgfmathsetmacro\angle{90-360/7*(\n-1)}
      \node[dot] (n\n) at (\angle:1) {};
    \fill (\angle:1)  circle (0.1);
  }
  \draw[thick] (n1) -- (n2) -- (n3) -- (n4) -- (n5) -- (n6) -- (n7) -- (n1);
  \coordinate [label=$C_7$] (G) at (0,-1.7);
\end{tikzpicture}
\caption{The graphs in Examples~\ref{examplecoloring} and~\ref{examplecritical}.}
\label{fig:cyclescoloring}
\end{center}
\end{figure}
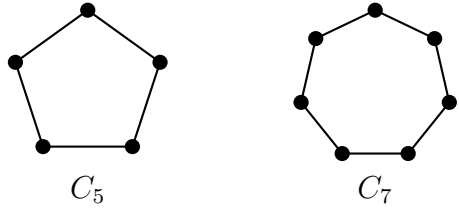
One may check that $\chi_{\mathrm c}(C_5)=3$ and that  $(2,0)$, $(1,1)$ and $(0,2)$ all are witnessing pairs for $C_5$. One may also check that $\chi_{\mathrm c}(C_7)=4$ and that the witnessing pairs for $C_7$ are $(3,0)$ and $(2,1)$, whereas $(1,2)$ and $(0,3)$ are covering pairs for $C_7$.
\end{example}

We will need to measure how much of a graph $H$ is left when we cover as much as possible of $H$ with $s$ cliques and $t$ independent sets. The next definition captures this idea:

\begin{definition}\label{defFF}
For a graph $H$ and non-negative integers $s$ and $t$, denote by $\mathcal{F}(H,s,t)$ the set of minimal (by induced containment) graphs $F$ such that $H$ can be covered by $s$~{cliques}, $t$~{independent} sets, and a copy of $F$. In other words, $\mathcal{F}(H,s,t)$ consists of the  graphs in the set
$$\{H- U \mid U\subseteq V_H\text{ and $H[U]$ is an $(s,t)$-template}\}$$
which are minimal with respect to induced containment.
\end{definition}

In particular, one gets $\mathcal{F}(H, s, t)=\{\emptyset\}$ if $H$ itself is an $(s,t)$-template.  Notice moreover that in practice when trying to determine what a specific $\mathcal F(H,s,t)$ is, we consider \emph{maximal} cliques and independent sets. If they were not maximal, then we would just end up with a graph $F$ that is not minimal, that is, a graph $F$ that strictly contains some other $F'$, obtained as an induced subgraph with maximal cliques and independent sets.

\begin{example}
Consider $H=P_5$. The set $\mathcal F(P_5,1,0)$ consists of two graphs $F_1$ and $F_2$: $F_1$ is the path on three vertices, obtained when we choose to cover one of the external edges of $P_5$, and $F_2$ is the the disjoint union of an edge and a vertex,
obtained when we choose to cover one of the internal edges of $P_5$. Indeed, neither $F_1$ nor $F_2$ is an induced subgraph of the other. 
\end{example}

\begin{definition}\label{defPP}
A graph $G$ is called \emph{$H$-free} if $G$ does not contain a copy of $H$ as an induced subgraph. For a family $\mathcal{F}$ of graphs, let 
$$\mathcal{P}(n,\mathcal{F}):=\{\text{graphs $G$ on $n$ vertices}\mid\text{for all $H\in \mathcal{F}$,  $G$ is $H$-free}\}.$$ 
If $\mathcal{F}=\{ H \}$ consists of a single graph, we simplify the notation to $\mathcal{P}(n,H)$.
\end{definition}

We finally come to the main definition of this section, from\til\cite{BaBu}:

\begin{definition}\label{def:critical}
A graph $H$  is \emph{critical} if, for all non-negative integers $s$ and $t$ with $s+t=\chi_{\mathrm c}(H)-2$ and for all large enough $n$, there are at most two graphs in $\mathcal{P}(n,\mathcal{F}(H,s,t))$.
\end{definition}

\begin{example}\label{examplecritical}
For the graphs in Figure~\ref{fig:cyclescoloring}, we know from Example~\ref{examplecoloring} that $\chi_{\mathrm c}(C_5)=3$ and $\chi_{\mathrm c}(C_7)=4$. To determine whether $C_5$ is critical, one needs to consider the pairs $(s,t)$ such that $s+t=3-2$, namely $(1,0)$ and $(0,1)$. We have
$$\mathcal{F}(C_5,1,0)=\{
\begin{tikzpicture}[>=latex, scale=0.5]
\draw (0,0)--(.4,.4)--(.8,0);
\fill (0,0) circle (0.1);
\fill (.4,.4) circle (0.1);
\fill (.8,0) circle (0.1);
\end{tikzpicture}
\},\qquad \mathcal{F}(C_5,0,1)=\{
\begin{tikzpicture}[>=latex, scale=0.5]
\draw[thick] (0,0)--(.8,0);
\fill (0,0) circle (0.1);
\fill (.4,.4) circle (0.1);
\fill (.8,0) circle (0.1);
\end{tikzpicture}
\}.$$
Then $C_5$ is not critical, because for large $n$ the set $\mathcal{P}(n,\mathcal{F}(C_5,1,0))$ consists of more than two elements: in particular it always contains at least the graph on $n$ vertices with no edges, the graph on $n$ vertices with one edge, and the complete graph $K_n$. On the other hand, $C_7$ is critical: as $\chi_{\mathrm c}(C_7)=4$, we need to inspect values of $s$ and $t$ such that $s+t=4-2$, which give
$$\mathcal{F}(C_7,2,0)=\{
\begin{tikzpicture}[>=latex, scale=0.5]
\draw (0,0)--(.4,.4)--(.8,0);
\fill (0,0) circle (0.1);
\fill (.4,.4) circle (0.1);
\fill (.8,0) circle (0.1);
\end{tikzpicture}\,,\, 
\begin{tikzpicture}[>=latex, scale=0.5]
\draw[thick] (0,0)--(.8,0);
\fill (0,0) circle (0.1);
\fill (.4,.4) circle (0.1);
\fill (.8,0) circle (0.1);
\end{tikzpicture}
\},\qquad\mathcal{F}(C_7,1,1)=\{
\begin{tikzpicture}[>=latex, scale=0.5]
\fill (0,0) circle (0.1);
\fill (.4,.3) circle (0.1);
\end{tikzpicture}
\},\qquad\mathcal{F}(C_7,0,2)=\{
\begin{tikzpicture}[>=latex, scale=0.6]
\fill[white] (0,0) circle (0.1);
\fill (0,.1) circle (0.1);
\end{tikzpicture}
\}.$$
Thus, for large $n$ we get
$$\mathcal{P}(n,\mathcal{F}(C_7,2,0))=\{K_n,\overline{K_n}\},
\quad\mathcal{P}(n,\mathcal{F}(C_7,1,1))=\{K_n\},
\quad\mathcal{P}(n,\mathcal{F}(C_7,0,2))=\emptyset$$
and therefore  $C_7$ is critical.
\end{example}

\begin{definition}
Let $\mathcal{A}(n)$ and $\mathcal{B}(n)$ be two families of graphs on $n$ vertices  such that $\mathcal{B}(n)\subseteq \mathcal{A}(n)$. We say that \emph{almost every graph in $\mathcal{A}(n)$ is in $\mathcal{B}(n)$} if
$$\lim_{n\to\infty} \frac{|\mathcal{A}(n)|}{|\mathcal{B}(n)|}=1.$$
We also write ``almost all'' or ``for almost all'', with the same connotation.
\end{definition}

%
%

The following is the main theorem of~\cite{BaBu} and one of the key tools for us.

\begin{lemma}[\cite{BaBu}, Theorem 1.9]\label{mainthmbb}
Let $H$ be a graph such that $\chi_{\mathrm c}(H)\ge3$. Then the following are equivalent:
\begin{enumerate}
\item almost every $H$-free graph  is an $(s,t)$-template, for some $s$ and $t$ such that  $(s,t)$ is a witnessing pair for $H$ (that is, $s+t=\chi_\mathrm c(H)-1$ and $(s,t)$ is not a covering pair for $H$);
\item $H$ is critical.
\end{enumerate}
\end{lemma}

In fact, the part of the equivalence above that we will use several times in this paper is the implication from part 2 to part 1. We recall a very famous classical result, where by \textit{triangle} we mean the cycle $C_3$:
\begin{theorem}[Erd\H{o}s--Kleitman--Rothschild, \cite{EKR73}]\label{thm:erdosandfriends}
Almost all triangle-free graphs are bipartite.
\end{theorem}

\begin{example}
Consider the triangle $C_3$. One may check that $\chi_{\mathrm c}(C_3)=3$ and that the only witnessing pair is $(0,2)$. The $(0,2)$-templates are exactly the bipartite graphs, and since we know that almost all $C_3$-free graphs are $(0,2)$-templates by Theorem~\ref{thm:erdosandfriends}, we deduce by Lemma~\ref{mainthmbb} that $C_3$ is critical. It is of course possible to show that $C_3$ is critical also directly by definition, as we did above for $C_7$. 
In the case of $C_7$, we observed that the witnessing pairs are $(3,0)$ and $(2,1)$, hence by Lemma~\ref{mainthmbb} almost every $C_7$-free graph is a $(3,0)$-template or a $(2,1)$-template.
\end{example}

In general, one goal of extremal graph theory is to understand (at least asymptotically) the typical structure of the graphs that exclude a given induced subgraph. The first folklore result in this line of research is Theorem\til\ref{thm:erdosandfriends} above. Lemma\til\ref{mainthmbb} generalizes that, and it describes the typical structure of graphs excluding a \emph{critical} induced subgraph in terms of cliques and independent sets. Cliques and independent sets constitute the simplest kind of graphs, and in this sense Lemma\til\ref{mainthmbb} settles this case by \emph{characterizing} the induced subgraphs for which such a simple description exists. (We remark once more that, despite this being the simplest case, the concept of a critical graph is quite technical.) Very little is known if the induced subgraph one wants to exclude is not critical, and in the known cases the ``typical structure'' has a more complicated description than in the case of critical graphs. For instance, Theorem\til1.1 of\til\cite{hcforgraphs} states that for $k\ge6$, almost all induced-$C_{2k}$-free graphs can be covered by $k-2$ cliques and a graph whose complement is a disjoint union of stars and triangles.

\subsection{Combinatorial commutative algebra: Betti numbers of edge ideals}\label{sec:backalgebra}

A \emph{simplicial complex} $\Delta$ on vertex set $V$ is a family of subsets of $V$ such that, whenever $\sigma\in\Delta$ and $\sigma'\subseteq\sigma$, we have $\sigma'\in\Delta$. We call the elements of $\Delta$  its \emph{faces}, and for a face $\sigma$ we say that  the \emph{dimension} of $\sigma$ is $\dim\sigma:=|\sigma|-1$. 
The simplicial complexes that we will consider in this paper are the following:

\begin{definition}
Let $G$ be a finite simple graph with vertex set $V$. The \emph{independence complex} of $G$, denoted $\mathrm{Ind}(G)$, is the simplicial complex with vertex set $V$ whose faces are the independent sets of $G$.
\end{definition}

One may define the \emph{reduced homology} of a simplicial complex $\Delta$ over a field $\mathbb{K}$. We refer to~\cite{MiSt} for a brief introduction. In short, denoting by $F_i$ the set of faces of $\Delta$ of dimension $i$ and defining suitable differentials, one gets a chain complex
$$0\longrightarrow\mathbb{K}^{F_{n-1}}\xrightarrow{\partial_{n-1}}
\dots\longrightarrow\mathbb{K}^{F_i}\stackrel{\partial_i}\longrightarrow\mathbb{K}^{F_{i-1}}\xrightarrow{\partial_{i-1}}\dots
\stackrel{\partial_1}\longrightarrow\mathbb{K}^{F_0}\stackrel{\partial_0}\longrightarrow \mathbb{K}^{F_{-1}}\longrightarrow 0,$$
and the \emph{$i$-th reduced homology} of $\Delta$ over  $\mathbb{K}$ is defined as
$$\tilde H_i(\Delta;\mathbb{K}):=\ker(\partial_i)/\textrm{im}(\partial_{i+1}).$$
In this paper we will fix a field $\mathbb K$ and write simply $\tilde H_i(\Delta)=\tilde H_i(\Delta;\mathbb{K})$. 
The homology computations in this paper are not particularly involved, especially in the context of our main results. The facts about homology that we use are in fact very standard and we refer for instance to\til\cite{MiSt} for the proofs.

\begin{definition}
Fixed a field $\mathbb K$, for a finite simple graph $G=(V_G,E_G)$, one defines the  \emph{edge ideal} of $G$ as
$$I_G:=(x_vx_w\mid \{v,w\}\in E_G)$$ inside the polynomial ring $S=\mathbb{K}
[x_v\mid v\in V_G]$. 
\end{definition}

\begin{remark}
Notice that we consider unlabeled graphs almost everywhere in the paper. The edge ideal of an unlabeled graph is not well defined, but the main results of this paper concern homological invariants of edge ideals, and these invariants are the same for any labeling of $G$. Therefore  we shall talk about the ``edge ideal of an unlabeled graph''.
\end{remark}

The \emph{(graded) Betti numbers} of a finitely generated, graded $S$-module $M$ are numerical invariants of $M$. For any $i\in\mathbb Z_{\ge0}$ and $j\in\mathbb{Z}$, they are written $\beta_{i,j}(M)$ or just $\beta_{i,j}$ if $M$ is clear from the context. This is not fundamental for the purposes of our paper, but for the sake of completeness we recall that the graded Betti numbers can be defined in terms of resolutions: the free modules in a minimal graded free resolution $\cdots \to F_2\to F_1\to F_0$ of $M$ can be written in a unique way as
$$F_i=\bigoplus_{j\in\mathbb Z}S(-j)^{\beta_{i,j}},$$
so that $\beta_{i,j}$ is the number of copies of $S$ shifted by $j$ in $F_i$. Alternatively, 
we recall that $\beta_{i,j}(M)=\dim_\mathbb K\mathrm{Tor}_i(M,\mathbb K)_j$, where 
$\mathrm{Tor}_i(-,\mathbb K)$ is the $i$-th left derived functor of $-\otimes_S\mathbb K$, and the dimension of the $j$-th graded piece of $\mathrm{Tor}_i(M,k)$ is taken as a $\mathbb K$-vector space. We refer the interested reader to\til\cite{Pe} for additional details.

The Betti numbers of $M$ are usually arranged in the so-called \emph{Betti table of $M$}, so that in column $i$ and row $j$ one puts the number $\beta_{i,i+j}$:
$$\begin{array}{c|ccccc}
& 0&1&2&3& \cdots\\
\hline
0&  \beta_{0,0} & \beta_{1,1} & \beta_{2,2} & \beta_{3,3}&\cdots\\
1&  \beta_{0,1} & \beta_{1,2} & \beta_{2,3} & \beta_{3,4}&\\
2&  \beta_{0,2} & \beta_{1,3} & \beta_{2,4} & \beta_{3,5}&\\
3&  \beta_{0,3} & \beta_{1,4} & \beta_{2,5} & \beta_{3,6}\\
\vdots&  \vdots &  &  &  & \ddots
\end{array}$$

The modules $M$  considered in this paper are always edge ideals $I_G$, and for this purpose one may consider the following special case of Hochster's formula (see\til\cite{MiSt}) as a definition of Betti numbers for edge ideals: for the Betti numbers in row $r$ of the table, one has
$$\beta_{i,r+i}(I_G)=\sum_{W\in\binom{V_G}{r+i}}\dim_\mathbb{K}\tilde H_{r-2}\big(\mathrm{Ind}(G)[W]\big).$$ 
In the case of edge ideals, many Betti numbers are known to vanish. First of all, the Betti numbers in rows $0$ and $1$ of the table are always zero, so the Betti numbers that we need to consider are on row $2$ onwards.
Secondly, by the \emph{main diagonal} of the Betti table we mean the diagonal consisting of the numbers $\beta_{i,\,2(i+1)}$, for $i\ge0$. It is well known that the Betti numbers to the left of the main diagonal are zero for edge ideals.

Note that one may also write the Betti numbers on the main diagonal as $\beta_{r-2,\,2(r-1)}$, where $r\ge2$ is the row index.
In general, the numbers on row\til$r$ can be written as $\beta_{r-2+p,\,2(r-1)+p}$, for some integer $p$ that shows how far horizontally that Betti number is from the main diagonal. 

It turns out that there are only  a finite number of non-zero entries in the Betti table, and the \emph{(Castelnuovo--Mumford) regularity} of $M$
$$\mathrm{reg}(M):=\max\{j\mid \beta_{i,i+j}(M)\ne0\text{ for some $i$}\}$$
is the highest index of a row with a non-zero entry.




\section{Parabolic clusters and parabolic Betti numbers}\label{sec:parabolic}

A graph is \emph{$k$-partite} if the vertex set can be partitioned into $k$ independent sets. Following the notation of~\cite{Dies}, we denote by $K_{a_1,a_2,\dots,a_k}$ the \emph{complete $k$-partite graph} whose independent sets have orders $a_1,a_2,\dots,a_k$ and where any two vertices in different independent sets are connected with an edge.  Recall that $K_a$ denotes the clique on $a$ vertices. We moreover denote by $\overline G$ the \emph{complement} of a graph $G$, that is, the graph on the vertex set of $G$ with exactly the edges that $G$ does not have. 

\begin{definition}\label{def:parabcluster}
A \emph{$k$-cluster} is the disjoint union of $k$ cliques, or equivalently the complement of a complete $k$-partite graph. If the number $k$ of cliques  is clear, we omit it. We denote a $k$-cluster by 
$$\overline{K_{a_1,\dots,a_k}}=K_{a_1}\sqcup K_{a_2}\sqcup\dots\sqcup K_{a_k},$$ 
where $a_i$ is the number of vertices in the $i$-th clique, and we assume $a_1\le a_2\le\dots\le a_k$. Let $k\ge2$.  If $a_1=2$ and $2\le a_i\le i$ for all $i\in\{2,3,\dots,k\}$, then we say that the $k$-cluster $\overline{K_{a_1,\dots,a_k}}$ is \emph{parabolic}.
\end{definition}

\begin{example}\label{ex:parabclustersforsmallk}
There is only one parabolic $2$-cluster: $\overline{K_{2,2}}$; two parabolic $3$-clusters: $\overline{K_{2,2,2}}$, $\overline{K_{2,2,3}}$; and five  parabolic $4$-clusters: $\overline{K_{2,2,2,2}}$, $\overline{K_{2,2,2,3}}$, $\overline{K_{2,2,2,4}}$,  $\overline{K_{2,2,3,3}}$, $\overline{K_{2,2,3,4}}$.
\end{example}

\begin{remark}
Our original proof of Proposition\til\ref{prop:clusterInTempV2} revolved around counting the parabolic $k$-clusters. Although we replaced that by a simpler argument, we think it is still worthwhile to mention that the number of parabolic $k$-clusters is the Catalan number $$C_{k-1}=\frac1k\binom{2(k-1)}{k-1}.$$
\end{remark}

%
%
%
%

Our next step is to relate $k$-clusters and Betti numbers.

\begin{lemma}\label{homologyoffatmatching}
Let $C$ be a $k$-cluster with each clique containing at least two vertices. Then
$$\tilde H_i(\mathrm{Ind}(C))\ne0\quad\text{if and only if}\quad i=k-1.$$
\end{lemma}

\begin{proof}
Notice that $C$ is a chordal graph and that $\mathrm{Ind}(C)$ is a pure complex of dimension $k-1$. The statement follows from the fact that the independence complex of a chordal graph is shellable, and hence it is homotopy equivalent to a wedge sum of spheres of that dimension (see\til\cite{DE09,VTV08,W09}).
\end{proof}

One may also prove the lemma above more directly by induction on $k$ and on the number of vertices: in the induction step, if $v$ is a vertex in some clique $K$ with more than two vertices, one may consider the Mayer--Vietoris long exact sequence
$$\xymatrix@=1.5em{
\dots\ar[r]^(.25){\partial_{i+1}}& H_i(\mathrm{Ind}(G- K))\ar[r]&  H_i(\mathrm{Ind}(G- v))\ar[r]&  H_i(\mathrm{Ind}(G))\ar[lld]^(.25){\partial_i} &\\
& H_{i-1}(\mathrm{Ind}(G- K))\ar[r]&  H_{i-1}(\mathrm{Ind}(G- v))\ar[r]&  H_{i-1}(\mathrm{Ind}(G)) \ar[r]^(.7){\partial_{i-1}}&\dots\\
}$$
from which the result follows.

\begin{lemma}\label{moregeneralthani2i}
Let $k\ge2$ and let $C=\overline{K_{a_1,\dots,a_k}}$ be a $k$-cluster with $a_i\ge2$ for all $i$. Denote $p=a_1+\dots+a_k-2k$. If a graph $G$ is such that
$$\beta_{k-1+p,\,2k+p}(I_G)=0,$$
then $G$ is $C$-free.
\end{lemma}

\begin{proof}
By Hochster's formula we have
$$\beta_{k-1+p,2k+p}(I_G)=\sum_{W\in\binom{V_G}{2k+p}}\dim_\mathbb{K} \tilde H_{k-1}\big(\mathrm{Ind}(G)[W]\big).$$
Notice that the $k$-cluster $C$ has $2k+p=a_1+\dots+a_n$ vertices, and by Lemma~\ref{homologyoffatmatching} we know that $\tilde H_{k-1}(\mathrm{Ind}(C))\ne0$.
\end{proof}

\begin{example}
We illustrate the phenomenon of Lemma~\ref{moregeneralthani2i} for \emph{parabolic} $k$-clusters with $2\le k\le 5$. 
We omit the commas in the subscripts of the clusters to lighten the notation:
$$\begin{array}{c||c|c|c|c|c|c|c|c|c|c}
&1&2&3&4&5&6&7&8&9&10\\
\hline
\hline
3&\beta_{1,4}&&&&&&&&\\
&\overline{K_{22}}&&&&&&&&\\
\hline
4&&\beta_{2,6}&\beta_{3,7}&&&&&&&\\
&&\overline{K_{222}}&\overline{K_{223}}&&&&&&&\\
\hline
5&&&\beta_{3,8}&\beta_{4,9}&\beta_{5,10}&\beta_{6,11}&&&\\
&&&\overline{K_{2222}}&\overline{K_{2223}}&\overline{K_{2224}}&\overline{K_{2234}}&&&\\
&&&&&\overline{K_{2233}}&&&&\\
\hline
6&&&&\beta_{4,10}&\beta_{5,11}&\beta_{6,12}&\beta_{7,13}&\beta_{8,14}
&\beta_{9,15}&\beta_{10,16}\\
&&&&&&&\overline{K_{22225}}&\overline{K_{22235}}
&\overline{K_{22245}}&\\
&&&&\overline{K_{22222}}&\overline{K_{22223}}&\overline{K_{22224}}&\overline{K_{22234}}
&\overline{K_{22244}}&\overline{K_{22245}}&\overline{K_{22345}}\\
&&&&&&\overline{K_{22233}}&\overline{K_{22333}}
&\overline{K_{22334}}&\overline{K_{22344}}&\\
\end{array}$$
By Lemma~\ref{moregeneralthani2i}, if a Betti number $\beta_{i,j}(I_G)$ written in the table above is zero, then $G$ is $C$-free, for any parabolic cluster $C$ written in the same cell as $\beta_{i,j}(I_G)$.
\end{example}

We give a name to the Betti numbers whose vanishing implies the absence of \emph{parabolic} $k$-clusters as induced subgraphs:

\begin{definition}\label{parabBnbs}
Let $r\ge3$. A Betti number $\beta_{i,i+r}$ on the $r$-th row of the Betti table is called  \emph{parabolic} if
$$r-2\le i\le r-2+\binom{r-1}2.$$
\end{definition}

Recall from Section~\ref{sec:backalgebra} that the Betti numbers on the $r$-th row of the Betti table can be written as $\beta_{r-2+p,\,2(r-1)+p}$, where $p$ represents the distance from the main diagonal $\beta_{0,2},\beta_{1,4},\beta_{2,6},\dots$ In several proofs it will be helpful to keep in mind that parabolic Betti numbers on the $r$-th row are exactly those of the form
$$\beta_{r-2+p,\,2(r-1)+p}\qquad\text{for $r\ge2$ and $0\le p\le\binom{r-1}2$.}$$
More explicitly, the region of the Betti table consisting of the parabolic Betti numbers is bounded by the main diagonal of numbers $\beta_{i,j}$ with
$$i=r-2\quad\text{and}\quad j=2(r-1),\qquad\text{for $r\ge3$,}$$
and the parabola consisting of the numbers $\beta_{i,j}$ with
\begin{align*}
i&=r-2+\binom{r-1}2=\frac{(r-3)(r-2)}2\\
j&=2(r-1)+\binom{r-1}2=\frac{(r-1)(r+2)}2,
\end{align*}
for $r\ge3$.
In Figure~\ref{fig:starredregion} 
the parabolic Betti numbers in the top-left portion of the Betti table are marked by gray squares.

\begin{remark}
The parabolic Betti numbers on the $k$-th row of the Betti table are related to the $k$-parabolic clusters as a special case of Lemma~\ref{moregeneralthani2i}. However, while that lemma holds for any cluster, we will only consider parabolic clusters since these have the property of being critical. This is discussed in the following section.
\end{remark}

\subsection{Parabolic clusters are critical}

\begin{lemma}\label{Hcritical}
Let $C$ be a parabolic $k$-cluster, with $k\ge2$. Then the following hold:
\begin{enumerate}
\item[(a)] $\chi_{\mathrm c}(C)=k+1$.
\item[(b)] $(k-1,1)$ is the only witnessing pair for $C$. 
\item[(c)] $C$ is critical.
\end{enumerate}
\end{lemma}

\begin{proof}
Write $C=\overline{K_{a_1,\dots,a_k}}$, where $a_1,\dots,a_k$ are the orders of the cliques of $C$.

(a) In order to  prove that $\chi_{\mathrm c}(C)=k+1$,  first of all we show that  if $s$ and $t$ are non-negative integers such that $s+t=k+1$, then $(s,t)$  is a covering pair for~$C$. We do this by induction on $k$. For $k=2$ there is only one cluster to consider: the matching consisting of two disjoint edges. One may check that $(3,0)$, $(2,1)$, $(1,2)$ and $(0,3)$ are covering pairs  in this base case. To continue the induction, let $k>2$. If~$s>1$, then $(s,t)$ is a covering pair for $C$ because $(s-1,t)$ is a covering pair for the $(k-1)$-cluster $K_{a_1}\sqcup\dots\sqcup K_{a_{k-1}}$ by the induction hypothesis: one is simply simultaneously adding a clique and a color in the induction step, and one may use that new color for the new clique. So we only need to check that the remaining pair~$(0,k+1)$ is a covering pair for $C$. Indeed, the cliques of $C$ have order at most $k$, so $k$ independent sets are already enough to cover even the largest possible parabolic $k$-clusters.

Thus, every pair $(s,t)$ with $s+t=k+1$ is a covering pair for $C$, which means that $\chi_{\mathrm c}(C)\le k+1$. In order to show that the equality holds, we need to exhibit a witnessing pair $(s,t)$ with $s+t=k$. For this we pick $(s,t)=(k-1,1)$. The pair $(k-1,1)$ is indeed not a covering pair for $C$, that is, we cannot cover $C$ with $k-1$ cliques and one independent set. Indeed, we may use the $k-1$ cliques at our disposal to cover at most only $k-1$ of the $k$ cliques of the cluster $C$. The remaining clique $K_{a_i}$ of $C$ cannot be covered by an independent set, because we have $a_i\ge2$ by definition of parabolic cluster.

(b) In the proof of part (a) we show that $(k-1,1)$ is a witnessing pair for $C$. Now we prove that any other pair $(s,t)$ with $s+t=k$ is a covering pair for $C$. The pair $(k,0)$  is a covering pair by definition of a $k$-cluster. And for any pair $(s,k-s)$, with $0\le s\le k-2$, one may cover the $s$ subgraphs $K_{a_{k-s+1}},\dots,K_{a_k}$ of $C$ with $s$ cliques, and cover the remaining cliques $K_{a_1},\dots,K_{a_{k-s}}$ with the $k-s$ independent sets, which are enough because $a_i\le k-s$ for $i\le k-2$, by definition of parabolic cluster.

(c) We need to consider the families of graphs $\mathcal{F}(C,s,t)$---see Definition~\ref{defFF}---for all pairs $(s,t)$ with $s+t=\chi_{\mathrm c}(C)-2=k-1$. The easiest cases are those of the pairs $(k-1,0)$ and $(k-2,1)$: if we cover as much of  $C$ as possible with $k-1$ cliques, we are left with the smallest clique of $C$, which is $K_2$. And if we cover as much of $C$ as possible with $k-2$ cliques and one independent set, we are left with two isolated vertices, which one may write as $\overline{K_2}$. These are the ways to cover $C$ with the pairs $(k-1,0)$ and $(k-2,1)$ that leave the least number of vertices uncovered. Indeed, all other ``residue graphs'' that are left after covering $C$ with the pairs $(k-1,0)$ or $(k-2,1)$ contain respectively $K_2$ and $\overline{K_2}$ as induced subgraphs. That is, we have
$$\mathcal{F}(C,k-1,0)=\{K_2\}\quad\text{and}\quad\mathcal{F}(C,k-2,1)=\{\overline{K_2}\}.$$
As for the  remaining cases, we claim that for any pair $(s,k-1-s)$, with $0\le s\le k-3$, one has
$$\mathcal{F}(C,s,k-1-s)\subseteq\{\emptyset,K_1\}.$$
We may cover the $s$ largest cliques $K_{a_{k-s+1}},\dots,K_{a_k}$, and we are left with $K_{a_1},\dots,K_{a_{k-s}}$.
With the $k-1-s$ independent sets at our disposal we may cover all of these, except when $a_{k-s}=k-s$, in which case we cover everything except one single vertex. So in this case $\mathcal{F}(C,s,k-1-s)=\{K_1\}$, and otherwise if $a_{k-s}<k-s$, then $\mathcal{F}(C,s,k-1-s)=\{\emptyset\}$. In order to show that $C$ is critical, one needs to check that, for all $s$ and $t$ with $s+t=\chi_{\mathrm c}(C)-2$ and for all large $n$, the family of graphs $\mathcal{P}(n,\mathcal{F}(C,s,t))$---see Definition~\ref{defPP}---contains at most two elements. By the discussion above, there are exactly four cases to check: for all large $n$, we have 
\begin{align*}
\qquad\mathcal{P}(n,\{K_2\})&=\{\overline{K_n}\},&\mathcal{P}(n,\{K_1\})&=\emptyset,\qquad\\
\mathcal{P}(n,\{\overline{K_2}\})&=\{K_n\},&\mathcal{P}(n,\{\emptyset\})&=\emptyset.
\end{align*}
All of these sets have cardinality smaller than $2$, so indeed $C$ is critical.
\end{proof}

\begin{theorem}\label{provedconjecture}
Let $k\ge2$ and let $H$ be a parabolic $k$-cluster. Then almost every $H$-free graph  is a $(k-1,1)$-template.
\end{theorem}

\begin{proof}
This follows directly from Lemma~\ref{mainthmbb} and Lemma~\ref{Hcritical}.
\end{proof}

\section{Results on $(d,1)$-templates}\label{sec:d1templates}

In this section we collect the last two main preliminary results  needed for our main theorems. Before that, we prove the following, which is a special case of an estimate by Balogh and Butterfield:

\begin{lemma}[\cite{BaBu}, Corollary 2.3]\label{lem:estimatetemplates}
For any positive integer $d$ and $\varepsilon >0$, there exists an integer $n_0$ such that
$$2^{(1-\frac1{d+1}-\varepsilon)\binom n2}<
\left|
\left\{
\begin{array}{c}
\textrm{$(d,1)$-templates $G$} \\
\textrm{on $n$ vertices}
\end{array}
\right\}\right|
<2^{(1-\frac1{d+1}-\varepsilon)\binom n2}$$
for all $n>n_0$.
\end{lemma}

\begin{proof}
The original statement of Corollary 2.3 of \cite{BaBu} is that for all graphs $H$ we have
$$2^{\big(1-\frac1{\chi_\mathrm c(H)-1}\big)\binom n2}\le|\mathcal Q(n,H)|<2^{\chi_\mathrm c(H)-1}
2^{\big(1-\frac1{\chi_\mathrm c(H)-1}\big)\binom n2}(\chi_\mathrm c(H)-1)^n,$$
where $\mathcal Q(n,H)$ is the family of graphs on $n$ vertices that are $(s,t)$-templates for some witnessing pair $(s,t)$ of $H$. If $H$ is a fixed parabolic $(d+1)$-cluster, this reduces to the desired estimate by parts (a) and (b) of Lemma~\ref{Hcritical}.
\end{proof}

Put in simple terms, the following proposition says that there are few $(d,1)$-templates where many of the $d$\til cliques are small.

\begin{proposition}\label{prop:clusterInTempV2}
For any positive integer $d$, almost every $(d,1)$-template contains all parabolic $k$-clusters, with $k\le d$, as induced subgraphs.
\end{proposition}

\begin{proof}
Very explicitly, we want to show that
\[
\lim_{n\to\infty}\frac{
\left|
\left\{
\begin{array}{c}
\textrm{$(d,1)$-templates $G$} \\
\textrm{on $n$ vertices}
\end{array}
\left|
\begin{array}{c}
\textrm{all parabolic $k$-clusters with $k\leq d$}  \\
\textrm{are induced subgraphs of $G$}
\end{array}
\right.
\right\} \right|}{
\big|\big\{
(d,1)\textrm{-templates on $n$ vertices}
\big\} \big|}
=1.
\]
Denote by $M_k$ the maximal parabolic $k$-cluster, namely $\overline{K_{2,2,3,4,\dots,k}}$, where each clique has the largest possible size. Observe that each parabolic $k$-cluster is an induced subgraph of\til$M_k$, and each $M_k$ is an induced subgraph of $M_d$ for all $k\le d$. Therefore, the set in the numerator, in the limit formula above, is equal to
$$\left\{
\begin{array}{c}
\textrm{$(d,1)$-templates $G$} \\
\textrm{on $n$ vertices}
\end{array}
\left|
\begin{array}{c}
\textrm{$M_d$ is an induced}  \\
\textrm{subgraph of $G$}
\end{array}
\right.
\right\}.$$
Hence, the statement is equivalent to the fact that
$$\lim_{n\to\infty}\frac{
\left|
\left\{
\begin{array}{c}
\textrm{$(d,1)$-templates $G$} \\
\textrm{on $n$ vertices}
\end{array}
\left|
\begin{array}{c}
\textrm{$M_d$ is \emph{not} an induced}  \\
\textrm{subgraph of $G$}
\end{array}
\right.
\right\} \right|}{
\big|
\big\{
(d,1)\textrm{-templates on $n$ vertices}
\big\} \big|}
=0.$$
And in turn this is equivalent to
$$\lim_{n\to\infty}\frac{
\left|
\left\{
\begin{array}{c}
\textrm{$(d,1)$-templates $G$} \\
\textrm{on $n$ vertices}
\end{array}
\left|
\begin{array}{c}
\textrm{$G$ is a}  \\
\textrm{$(d-1,1)$-template}
\end{array}
\right.
\right\} \right|}{
\big|
\big\{
(d,1)\textrm{-templates on $n$ vertices}
\big\} \big|}
=0,$$
since almost every $M_d$-free graph is a $(d-1,1)$-template, by Theorem~\ref{provedconjecture}. The quantity in the limit can be bounded above by
$$\frac{\big|\big\{
(d-1,1)\textrm{-templates on $n$ vertices}
\big\} \big|}
{\big|\big\{
(d,1)\textrm{-templates on $n$ vertices}
\big\} \big|}\le\frac{2^{(1-\frac1d+\varepsilon)\binom n2}}{2^{(1-\frac1{d+1}-\varepsilon)\binom n2}}=2^{(-\frac1{d(d+1)}+2\varepsilon)\binom n2},$$
where the inequality holds by Lemma~\ref{lem:estimatetemplates}, and this tends to zero if we pick $\varepsilon<\frac1{2d(d+1)}$.
\end{proof}

The following is Theorem\til2 of\til\cite{W14}, and as stated there it follows from Theorem\til1.2 of\til\cite{KM06}.

\begin{lemma}\label{lem:Russ}
If\/ $G$ is a $(d,1)$-template, then $\mathrm{reg}(I_G)\le d+1$.
\end{lemma}


\section{Main results}\label{sec:main}

\begin{theorem}\label{thm:main1}
Let $\beta_{i,j}$ be a parabolic Betti number on the $r$-th row of the Betti table, for some $r\ge3$. Then almost every graph $G$ with $\beta_{i,j}(I_G)=0$ is an $(r-2,1)$-template.
\end{theorem}

\begin{proof}
By the discussion following Definition~\ref{parabBnbs}, the parabolic Betti number $\beta_{i,j}$ in the statement can be written in the form $\beta_{r-2+p,\,2(r-1)+p}$, for a non-negative integer $p\le\binom{r-1}2$. 
Let $H$ be any parabolic $(r-1)$-cluster of order $2(r-1)+p$. 
For any $n$, consider the families
\begin{align*}
\mathcal{B}(n)&:=\{\text{graphs $G$ on $n$ vertices with $\beta_{r-2+p,\,2(r-1)+p}(I_G)=0$}\},\\
\mathcal{H}(n)&:=\mathcal P(n,H)=\{\text{graphs on $n$ vertices that are $H$-free}\},\\
\mathcal{T}(n)&:=\{\text{graphs on $n$ vertices that are $(r-2,1)$-templates}\}.
\end{align*}
We will show that almost every $G\in\mathcal{B}(n)$ is in $\mathcal{T}(n)$. 
By Lemma~\ref{lem:Russ}, if $G$ is an $(r-2,1)$-template, then $\mathrm{reg}(I_G)\le r-1$. Therefore any Betti number on the $r$-th row of the Betti table, and in particular the number that we are considering in the statement, is zero. Hence we have the inclusion $\mathcal{T}(n)\subseteq\mathcal{B}(n)$, implying that $|\mathcal B(n)|/|\mathcal T(n)|\ge1$. 
By Lemma~\ref{moregeneralthani2i}, the vanishing of $\beta_{r-2+p,2(r-1)+p}(I_G)$ implies that $G$ is $H$-free, that is, $\mathcal{B}(n)\subseteq\mathcal{H}(n)$. In particular,  $|\mathcal{B}(n)|/|\mathcal{H}(n)|\le1$. 
Therefore we have
$$1\le\frac{|\mathcal{B}(n)|}{|\mathcal{T}(n)|}=\frac{|\mathcal{H}(n)|}{|\mathcal{T}(n)|}\frac{|\mathcal{B}(n)|}{|\mathcal{H}(n)|}\le\frac{|\mathcal{H}(n)|}{|\mathcal{T}(n)|}.$$
By Theorem~\ref{provedconjecture}, almost every $H$-free graph is an $(r-2,1)$-template, which means that 
$$\lim_{n\to\infty}\frac{|\mathcal{H}(n)|}{|\mathcal{T}(n)|}=1.$$
So by the squeeze theorem
$$\lim_{n\to\infty}\frac{|\mathcal{B}(n)|}{|\mathcal{T}(n)|}=1,$$
and this concludes the proof.
\end{proof}

\begin{theorem}\label{thm:main2}
Let $r\ge 3$ and fix a parabolic Betti number $\beta_{i,i+r}$. For almost every graph\til$G$ with $\beta_{i,i+r}(I_G)=0$,
\begin{enumerate}
\item $\reg(I_G)=r-1$, and
\item every parabolic Betti number of\/ $I_G$ above row $r$ is non-zero.
\end{enumerate}
\end{theorem}

\begin{proof}
Define the following sets:
\begin{align*}
\mathcal B(n)&:=\big\{\text{graphs $G$ on $n$ vertices}\mid \beta_{i,i+r}(G)=0\big\},\\
\mathcal R(n)&:=\big\{\text{graphs $G$ on $n$ vertices}\mid \reg(I_G)=r-1\big\},\\
\mathcal P(n)&:=\bigg\{\begin{array}{c}\text{graphs $G$}\\\text{on $n$ vertices}\end{array}\bigg| \begin{array}{c}\text{$\beta_{i,i+r}(I_G)=0$ and all parabolic Betti}\\
\text{numbers of $I_G$ above row $r$ are non-zero}
\end{array}\bigg\}.
\end{align*}
We clearly have $\mathcal R(n)\subseteq\mathcal B(n)$ and $\mathcal P(n)\subseteq B(n)$. The two parts of the statement mean respectively that
$$\lim_{n\to\infty}\frac{|\mathcal B(n)|}{|\mathcal R(n)|}=1\qquad\text{and}\qquad
\lim_{n\to\infty}\frac{|\mathcal B(n)|}{|\mathcal P(n)|}=1.$$
We first prove the second part and start by defining more sets:
\begin{align*}
\mathcal K(n)&:=\bigg\{\begin{array}{c}\text{graphs $G$}\\\text{on $n$ vertices}\end{array}\bigg| \begin{array}{c}\text{$G$ is an $(r-2,1)$-template and all parabolic $k$-clusters}\\
\text{with $k\le r-2$ are induced subgraphs of $G$}
\end{array}\bigg\},\\
\mathcal T(n)&:=\big\{\text{graphs $G$ on $n$ vertices}\mid \text{$G$ is an $(r-2,1)$-template}\big\}.
\end{align*}
Written in formulas, Theorem\til\ref{thm:main1} and Proposition\til\ref{prop:clusterInTempV2} state respectively that $\lim_{n\to\infty}\frac{|\mathcal B(n)|}{|\mathcal T(n)|}=1$ and $\lim_{n\to\infty}\frac{|\mathcal T(n)|}{|\mathcal K(n)|}=1$. By Lemma\til\ref{moregeneralthani2i}, we have $\mathcal K(n)\subseteq\mathcal P(n)$. But then we get
$$\lim_{n\to\infty}\frac{|\mathcal B(n)|}{|\mathcal K(n)|}=\lim_{n\to\infty}\frac{|\mathcal T(n)|}{|\mathcal K(n)|}\frac{|\mathcal B(n)|}{|\mathcal T(n)|}=1.$$ The second point in the statement then follows, as $1\le\frac{|\mathcal B(n)|}{|\mathcal P(n)|}\le \frac{|\mathcal B(n)|}{|\mathcal K(n)|}$.
Next, define
\begin{align*}
\mathcal R_\le(n)&:=\big\{\text{graphs $G$ on $n$ vertices}\mid \reg(I_G)\le r-1\big\},\\
\mathcal R_\ge(n)&:=\big\{\text{graphs $G$ on $n$ vertices}\mid \reg(I_G)\ge r-1\text{ and }\beta_{i,i+r}(I_G)=0\big\}.
\end{align*}
Note that $\mathcal R_\le(n)\cap\mathcal R_\ge(n)=\mathcal R(n)$ and $\mathcal R_\le(n)\cup\mathcal R_\ge(n)=\mathcal B(n)$ by construction, so that $|\mathcal R(n)|=|\mathcal R_\le(n)|+|\mathcal R_\ge(n)|-|\mathcal B(n)|$. Moreover we have the inclusions $\mathcal P(n)\subseteq\mathcal R_\ge(n)\subseteq\mathcal B(n)$ and $\mathcal R_\le(n)\subseteq\mathcal B(n)$.
And on the other hand, by Lemma\til\ref{lem:Russ}, we know that $\reg(I_G)\le r-1$ for any $(r-2,1)$-template $G$, namely $\mathcal T(n)\subseteq\mathcal R_\le(n)$.
In order to prove the first point in the statement, observe that
$$\frac{|\mathcal R(n)|}{|\mathcal B(n)|}=\frac{|\mathcal R_\le(n)|}{|\mathcal B(n)|}+\frac{|\mathcal R_\ge(n)|}{|\mathcal B(n)|}-\frac{|\mathcal B(n)|}{|\mathcal B(n)|}.$$
The first summand has limit equal to $1$ because $1\le\frac{|\mathcal B(n)|}{|\mathcal R_\le(n)|}\le\frac{|\mathcal B(n)|}{|\mathcal T(n)|}$ and by Theorem\til\ref{thm:main1}. The second summand has limit equal to $1$ because $1\le\frac{|\mathcal B(n)|}{|\mathcal R_\ge(n)|}\le\frac{|\mathcal B(n)|}{|\mathcal P(n)|}$ and by the second point in the statement, proven above.
\end{proof}

\begin{remark}
At first sight Theorem\til\ref{thm:main2} might look a bit counterintuitive, and it even seems to lead to a contradiction, as follows. One may think that for every graph $G$ with $\reg(I_G)=r-1$, all the parabolic Betti numbers below row $r$ vanish, in particular on some row $s>r$. But then these graphs $G$ satisfy the hypotheses of the theorem, and this would imply that $\reg(I_G)=s-1>r-1$, a contradition. This does not work because there are no graphs ``satisfying the hypotheses of the theorem''. We would like to stress the order of the quantifiers: the theorem is not about specific graphs; first we fix a parabolic Betti number, and the rest of the statement is an asymptotic result about that Betti number. (One can indeed make examples of graphs $G$ that have some vanishing parabolic Betti number on row $r$, but whose regularity is much smaller or much higher than $r$---the following section is dedicated to that.) 
More explicitly, consider the parabolic Betti numbers $\beta_1:=\beta_{i,i+r}$ and $\beta_2:=\beta_{j,j+s}$, respectively on rows $r$ and $s$, for some $r<s$. The sets involved in the first part of the statements of Theorem\til\ref{thm:main2} are
\begin{align*}
\mathcal B_1(n)&:=\big\{\text{graphs $G$ on $n$ vertices}\mid \beta_{i,i+r}(G)=0\big\},\\
\mathcal R_1(n)&:=\big\{\text{graphs $G$ on $n$ vertices}\mid \reg(I_G)=r-1\big\},\\
\mathcal B_2(n)&:=\big\{\text{graphs $G$ on $n$ vertices}\mid \beta_{j,j+s}(G)=0\big\},\\
\mathcal R_2(n)&:=\big\{\text{graphs $G$ on $n$ vertices}\mid \reg(I_G)=s-1\big\}.
\end{align*}
where the notation mirrors the one used in the proof. The statement itself, applied separately to the two distinct Betti numbers, explicitly means that
$$\lim_{n\to\infty}\frac{|\mathcal B_1(n)|}{|\mathcal R_1(n)|}=1\qquad\text{and}\qquad\lim_{n\to\infty}\frac{|\mathcal B_2(n)|}{|\mathcal R_2(n)|}=1.$$
The point is that they are two distinct statements: the sets $\mathcal R_1(n)$ and $\mathcal R_2(n)$ are clearly disjoint, and in general also the sets $\mathcal B_1(n)$ and $\mathcal B_2(n)$ are not related, in the sense that none is a subset of the other. In \emph{special} cases there are relations: for instance if the Betti numbers are both chosen to be on the main diagonal, then $\mathcal B_1(n)$ would be a subset of $\mathcal B_2(n)$, but this is indeed a special case. 
One could in principle also compare the size of $\mathcal B_1(n)$ and $\mathcal B_2(n)$, by using classical estimates by Pr\"omel and Steger (written in our notation as Theorem\til1.10 of\til\cite{BaBu}), but perhaps that would only add confusion here. Regardless of their size, the two sets are involved in unrelated statements.
\end{remark}


\section{How often is the regularity not $r-1$?}\label{sec:notrminus1}

According to our main result, Theorem~\ref{thm:main2}, the regularity of $I_G$ is $r-1$ for almost all graphs $G$ with a vanishing parabolic Betti number $\beta_{i,j}(I_G)=0$ on row $r$. It is natural to ask if the ``almost all'' in our main result is simply an artifact of our proof technique, and in fact the statement holds for all graphs. To see that this is not the case, one may for instance consider  a matching consisting of disjoint edges. In this case the non-zero Betti numbers lie on the main diagonal, so that there are many vanishing parabolic Betti numbers on many rows $r$, but the regularity can be very far from\til$r-1$. This gives a rather small set of counterexamples (containing one graph on $2n$ vertices for every $n$); this section is devoted to larger classes of counterexamples.

In this section it is demonstrated that there are many graphs whose edge ideal has regularity different from $r-1$ and with some vanishing parabolic Betti number on row\til$r$. We show that for large $n$, there are at least $2.99^n$ such graphs on $n$ vertices.

For our constructions we will need to estimae the number of unlabeled trees on $n$ vertices. This number was determined asymptotically by Otter~\cite{O48}, and this is today a textbook exercise in singularity analysis \cite{FS09}.
The following is a slightly suboptimal simplified lower bound.

\begin{lemma}\label{lem:countTrees}
There are at least $2.995^n$ trees on $n$ vertices, for large $n$.
\end{lemma}

\begin{lemma}\label{lem:indSpecEx}
For integers $a \geq 3$ and $b,c\geq 0$, let
\[ G = \overline{C_a \sqcup T_b} \sqcup {M_c},  \]
where $C_a$ is the cycle on $a$ vertices, $T_b$ is a tree on $b$ vertices, and $M_c$ is the perfect matching on $c$ edges.
\begin{enumerate}
\item[(1)] If $a+2c \leq m \leq a+b+2c$, then there is an induced subgraph $H$ of $G$ on $m$ vertices with $\dim \tilde{H}_{c+1}(\mathrm{Ind}(H))=1$;
\item[(2)] If $m< a+2c$, then, for any induced subgraph $H$ of $G$ on $m$ vertices, we have $\dim \tilde{H}_{c+1}(\mathrm{Ind}(H))=0$;
\item[(3)] For any induced subgraph $H$ of $G$, we have  $\dim \tilde{H}_{c+i}(\mathrm{Ind}(H))=0$ for $i>1$.
\end{enumerate}
\end{lemma}
\begin{proof}


For any graph $F$, the complex $\mathrm{Ind}(F \sqcup \textrm{edge})$ is the suspension of $\mathrm{Ind}(F)$ and hence has the same homology  shifted by one position. By iteration, the independence complex of the disjoint union of $F$ and a perfect matching on $e$ edges has the homology of the independence complex of $F,$ but shifted by $e$ positions.

We prove (2) and (3): Let $H$ be an induced subgraph of $G = \overline{C_a \sqcup T_b} \sqcup {M_c}.$ If $H \cap {M_c}$ has an isolated vertex, then so does $H$, and hence 
$\mathrm{Ind}(H)$ is contractible and in particular has no homology. Otherwise, $H \cap {M_c}$ is a perfect matching on $c'$ edges, and $\mathrm{Ind}(H)$ has the homology of $\mathrm{Ind}(H \cap  \overline{C_a \sqcup T_b})$ shifted upwards by $c'$ degrees. The homology of $\mathrm{Ind}(H \cap  \overline{C_a \sqcup T_b})$ is zero in dimension $i>1$, which establishes statement (3) as $c'\leq c$. To get $\dim \tilde{H}_{c+1}(\mathrm{Ind}(H)) > 0$ we need that $\overline{C_a} \subseteq H$ and $c=c'$. In particular, there will be at least $a+2c$ vertices in $H$, establishing statement (2).

Finally, we prove (1): For  $a+2c \leq m \leq a+b+2c$, construct $H$ by only removing vertices from $G$ that are in $\overline{T_b}$. By the preceding homology calculation,  $\dim \tilde{H}_{c+1}(\mathrm{Ind}(H)) = 1$.
\end{proof}

In the following result, we show that there exist some graphs $G$ with vanishing Betti numbers in row $r$ (condition (2)) and such that $\reg(I_G)=r$ (conditions (1) and (3)).

\begin{theorem}
Let $r\geq 3$ and $i \geq 2r- 4$ be integers. For large $n$, there are at least $2.99^n$ graphs $G$ on $n$ vertices such that
\begin{enumerate}
\item[(1)] on row $r$, one has $\beta_{j,r+j}(I_G)>0$ for $i < j \leq n-r$;
\item[(2)]  on row $r$, one has $\beta_{j,r+j}(I_G)=0$ for $j \leq i$;
\item[(3)] below row $r$, all Betti numbers are zero.
\end{enumerate}
\end{theorem}

\begin{proof}
Consider graphs $G = \overline{C_{i-2r+7} \sqcup T_{n-i+r-4}} \sqcup {M_{r-3}}$
for $n\geq i+r-4$. By Hochster's formula
\[ \beta_{j,r+j}(I_G)=\sum_{W\in\binom{V_G}{r+j}}\dim \tilde H_{r-2}(\mathrm{Ind}(G)[W]) \]
and by Lemma~\ref{lem:indSpecEx} with some elementary index calculations, the statements (1), (2) and (3) are true.
By Lemma~\ref{lem:countTrees}, there are at least $2.995^{n-i+r-4}$ such graphs for large $n$ by modifying $T_{n-i+r-4}$, and thus there are $2.99^n$ of them for large $n$.
\end{proof}

\subsection{The case of regularity $2$}

As an anonymous referee very kindly suggested, one could consider counterexamples coming from the very famous theorem by Fr\"oberg stated below. Indeed there are many graphs whose edge ideal has some vanishing parabolic Betti numbers on rows with high index, but the regularity is equal to\til$2$. In this section we collect well known results that combine into asymtpotics for such graphs.

The two following results are both very classical in the respective theories. We recall that ``split graph'' is the classical name for a $(1,1)$-template, and chordal graphs are graphs that do not contain induced cycles of length strictly greater than $3$.

\begin{theorem}[Bender et al.\til\cite{BRW}]\label{thm:almchosplit}
Almost all chordal graphs are split.
\end{theorem}

\begin{theorem}[Fr\"oberg\til\cite{F90}]\label{thm:frobg}
One has $\reg(I_G)=2$ iff the complement of $G$ is chordal.
\end{theorem}

For the readers who are more inclined towards combinatorics than commutative algebra, we recommend the proof in\til\cite{DE09} of Theorem\til\ref{thm:frobg}.
It is well known that for higher regularity one cannot hope for a combinatorial description similar to that of Theorem\til\ref{thm:frobg}, due to the fact that the characteristic of the ground field might affect the regularity. There have been attempts to generalize this result at least partially to hypergraphs, but we do not wish to go in that direction.

Denote by $s_n$ the number of split graphs on $n$ vertices.
Asymptotics for $s_n$ (and equivalently for the number of chordal graphs, by Theorem\til\ref{thm:almchosplit}) were already given in Corollary\til1  of\til\cite{BRW}:  in particular $\log_2 s_n$ tends to $n^2/4$. The following is a much more recent estimate:
\begin{theorem}[Troyka\til\cite{Troyka}]
The number $s_n$ of split graphs on $n$ vertices is asymptotically equal to
$$\widetilde s_n:=\frac1{n!}\bigg(\sum_{k=0}^n\binom nk 2^{k(n-k)}-n\sum_{k=0}^{n-1}\binom{n-1}k2^{k(n-1-k)}\bigg).$$
\end{theorem}

By combining the results above, one gets the following:

\begin{corollary}
The number of graphs $G$ on $n$ vertices with $\reg(I_G)=2$ is asymptotically equal to $\widetilde s_n$.
\end{corollary}

We therefore have $\widetilde s_n$ graphs on $n$ vertices with vanishing parabolic Betti numbers on row $r$, for arbitrarily high $r$, but with regularity equal to $2$.


\section{Non-parabolic Betti numbers}\label{sec:non-parabolic}

Our main theorems concern \emph{parabolic} Betti numbers. One may wonder if the situation can be reasonably generalized to any Betti number. Here we show what goes wrong. 

In Section\til\ref{sec:no-bound} we consider the Betti numbers on row $2$ of the Betti table: not only our methods with critical graphs cannot be applied, we actually show that the vanishing of a non-parabolic Betti number in row\til$2$ does not imply any constant bound on the regularity.

In Section\til\ref{sec:BNrow3} we show that the situation gets quite messy in the case of the first two non-parabolic Betti numbers on row\til$3$.

\subsection{Non-parabolic Betti numbers on row $2$: absence of constant bounds on the regularity}\label{sec:no-bound}

\begin{definition}
Let $G$ be a graph. The \emph{induced matching number} of $G$, denoted $\iota(G)$, is the number of edges in a largest induced matching of $G$.
\end{definition}

Note that $\iota(G)$ is not the \emph{matching number} of $G$, which is the number of edges in a largest arbitrary matching, not necessarily induced. For instance, in a square the matching number is\til$2$, whereas $\iota(G)=1$. The matching number and the induced matching number are both well-known invariants, including in the commutative algebra literature. A folklore lower bound for the regularity, easily proven via Hochster's formula, is $\iota(G)\le\reg(S/I_G)$. We mentioned this for the sake of completeness, although we do not use it in our proofs.

\begin{lemma}\label{lem:enumMatch}
If $G$ is a graph on $e$ edges and of maximal degree $d>0,$ then $\iota(G) \geq \frac{e}{2d^2}$.
\end{lemma}

\begin{proof}
We describe an algorithm that produces an induced matching. Start by coloring all edges black and set $I=\emptyset$. Pick a black edge and add it to $I$. Color that edge, all edges that are incident to that one, and all edges that are incident to those ones, in red. Pick a new black edge, add it to $I$ and repeat the procedure until all edges are red.

The edges in $I$ form an induced matching. Denoting by  $|I|$ the cardinality of $I$, the algorithm ends after $|I|$ steps, and at most $1+2(d-1)+2(d-1)(d-1)<2d^2$ edges are colored red in each step. Therefore, one has $e\le|I|2d^2$.

Observe that every induced matching, and in particular those of largest size, can be found by this algorithm. The point of the algorithm is to bound the number of red edges at each step.
\end{proof}

\begin{lemma}\label{lem:eMatch}
Fix $k\geq 3$. For any natural number $n$, denote
$$\mathcal{G}_n= \bigg\{
\begin{array}{c} \textrm{graphs $G$ on } \\ \textrm{$n$ vertices} \end{array}
 \bigg|  \begin{array}{c} \text{for all induced subgraphs $H$ of $G$ of order $k$,} \\ \textrm{the complement of $H$ is connected} \end{array}  \bigg\}.$$
Then 
$$\frac1{| \mathcal{G}_n |} \sum_{G \in  \mathcal{G}_n } \iota(G) \geq \frac{n-1}{4 (k-2)^2}.$$
\end{lemma}
\begin{proof}
For a graph $G,$ let $G'$ be the induced subgraph of $G$ consisting of the connected components of $G$ on at least three vertices. Introduce an equivalence relation on $\mathcal{G}_n$ by $G_1 \sim G_2$ if
$G_1'$ and $G_2'$ are isomorphic (actually {\em the same}, as we deal with unlabeled graphs).

We now describe the equivalence classes. Take a class $\mathbf{G}\in\mathcal{G}_n/\sim$ and let $G'$ be the graph on $m$ vertices that is the same for any $G\in\mathbf{G}.$ Then the equivalence class $\mathbf{G}$ consists of the graphs 
$$ G'\   \sqcup\  (\text{matching on $i$ edges})\  \sqcup\   (\text{$n-m-2i$ isolated vertices})$$
for $i=0,1,\ldots, \lfloor (n-m)/2 \rfloor$. In particular,
$$\sum_{G \in \mathbf{G} } \iota(G) = \sum_{i=0}^{\lfloor(n-m)/2\rfloor} (\iota(G')+i)\\
= | \mathbf{G} |  \bigg( \iota(G') +  \frac{ \big\lfloor \frac{n-m}{2} \big\rfloor}{2} \bigg).$$
Next we estimate $ \iota(G')$. There are at least $\frac{2}{3}m'$ edges in each connected component of $G'$ on $m'$ vertices, as $m'\geq 3$. Thus, in total, there are at least $\frac{2}{3}m$ edges in $G'$.

The maximal degree $d$ of $G'$ is at most $k-2$, because if there was a vertex of $G'$ with $k-1$ neighbours, then the induced subgraph of that vertex and those neighbours would not have a connected complement. By Lemma~\ref{lem:enumMatch}, 
\[ \iota(G') \geq \frac{\# \text{edges in }G'}{2d^2}\geq\frac{\frac{2}{3}m }{ 2(k-2)^2 } = \frac{m}{3 (k-2)^2} \]
whenever $G'$ is non-empty. If $G'$ is empty, then $\iota(G') =  \frac{m}{3 (k-2)^2} =0$. 
Estimate
\[
\frac{m}{3 (k-2)^2}+ \frac{ \big\lfloor \frac{n-m}2 \big\rfloor}2 \geq \frac{m}{4 (k-2)^2} + \frac{  \frac{n-m}2 -  \frac12 }{2(k-2)^2} = \frac{n-1}{4 (k-2)^2}
\]
to get
\[  \sum_{G \in \mathbf{G} } \iota(G) \geq  \frac{  | \mathbf{G} |  (n-1) }{4 (k-2)^2},
\]
and altogether
\begin{align*}
\frac1{| \mathcal{G}_n |}\sum_{G \in  \mathcal{G}_n } \iota(G) &= \frac1{| \mathcal{G}_n |}\sum_{\mathbf{G} \in  \mathcal{G}_n / \sim }  \sum_{G \in \mathbf{G} }  \iota(G)  \\
&\geq \frac1{| \mathcal{G}_n |} \sum_{ \mathbf{G} \in  \mathcal{G}_n / \sim }     \frac{|\mathbf{G}|(n-1) }{4 (k-2)^2}  =  \frac{n-1}{4 (k-2)^2}  .
\end{align*}
\end{proof}


Recall that the Betti numbers of the form $\beta_{k-2,\,k}$ are on row $2$ of the Betti table. The first one of them is $\beta_{0,2}$, attained for $k=2$. Its vanishing is trivial, it means that the graph has no edges. Therefore we consider $k\ge3$ below. Put in simple terms, the following result states that for the Betti numbers on row\til$2$ there is no analog of our Theorem\til\ref{thm:main2}: if it were possible to generalize that theorem to the Betti numbers on row\til$2$,  for almost all graphs $G$ with $\beta_{k-2,\,k}(I_G)=0$ we would have $\reg(I_G)=2-1=1$, which does not really make sense as all Betti numbers on row\til$1$ vanish to begin with. But even if one relaxes the statement and just asks for a function bounding the regularity, we prove that such a function would diverge.

%
%
%

\begin{theorem}
Let $k \geq 3$ and $r\colon\mathbb{N}\rightarrow\mathbb{N}$ be a function. For any natural number $n$, denote
\begin{align*}
\mathcal G(n)&:=\{\text{graphs $G$ on $n$ vertices} \mid \beta_{k-2,\, k}(I_G)=0\},\\
\mathcal R(n)&:=\{\text{graphs $G$ on $n$ vertices} \mid \beta_{k-2,\, k}(I_G)=0 \text{ and }\reg(I_G) \leq r(n)\}.
\end{align*}
If $\lim_{n\rightarrow \infty} \frac{|\mathcal R(n)|}{|\mathcal G(n)|} = 1$, then $\lim_{n\rightarrow \infty} r(n) = \infty$.
\end{theorem}

\begin{proof}
Assume on the contrary that $r(n)$ does not diverge. Then there exists some constant $K>0$ such that for every $n'$ there is an $n \ge n'$ with $r(n) \le K$.
Moreover, for every $\varepsilon > 0$ there is an $n'$ with 
$$\frac{|\mathcal R(n)|}{|\mathcal G(n)|} > 1- \varepsilon$$
for all $n \geq n'$, as the quotient is assumed to converge to\til$1$. 
Combining these two facts repeatedly produces an infinite sequence $n_1 < n_2 < n_3 < \dots$ of integers satisfying
$$\frac{ | \{ G \in \mathcal{G}(n_i) \mid  \mathrm{reg}(I_G) \leq K \} |}{ |  \mathcal{G}(n_i) |} > 1 - \varepsilon.$$
We may assume that $K<n_1$ to simplify the next computation. The regularity of any ideal $I_G$ associated to $G \in \mathcal{G}(n_i)$ is at most $n$, and thus,
$$\frac{ 1}{|\mathcal{G}(n_i)|} \sum_{G \in \mathcal{G}(n_i)}\mathrm{reg}(I_G) \leq (1- \varepsilon) K + \varepsilon n_i.$$
Next we produce a sum that contradicts the one above.
For a graph $G$, by Hochster's formula, the fact that $\beta_{k-2,\,k}(I_G) = 0$ is equivalent to the fact that the complement of $H$ is connected, for all induced subgraphs $H$ of $G$ order $k$. So, the set $\mathcal{G}(n)$ is the same as $\mathcal G_n$ in Lemma\til\ref{lem:eMatch}. According to that lemma,  the average order of maximal induced matchings of graphs $G$ on $n$ vertices with $\beta_{k-2,\,k}(I_G) = 0$ is at least $\frac{n-1}{4(k-2)^2}$. It is well known that an induced matching on $e$ edges has as independence complex a hyper-octahedron with the homology of a sphere of dimension $e-1$. Therefore an induced matching on $e$ edges in a graph $G$ forces the first $e$ Betti numbers on the main diagonal of the Betti table of $I_G$ to be non-zero, and in particular the regularity of $I_G$ is at least $e$. Thus,
$$\frac{n-1}{4(k-2)^2}\leq \frac{ 1}{|\mathcal{G}(n)|}\sum_{G \in \mathcal{G}(n) }   \mathrm{reg}(I_G)$$
for all $n$. Combining the two preceding inequalities yields that
$$\frac{n_i-1}{4(k-2)^2} \leq (1- \varepsilon) K + \varepsilon n_i$$
for an infinite sequence $n_1 < n_2 < n_3 < \dots$ of integers. Now set $\varepsilon := \frac{1}{5(k-2)^2}$ to get the desired contradiction. 
\end{proof}

\subsection{Non-parabolic Betti numbers on row $3$}\label{sec:BNrow3}

Our Theorem~\ref{thm:main2} concerns parabolic Betti numbers on row $r\ge3$.
In this section we illustrate what happens when we consider the first Betti numbers in row\til$3$ that are not parabolic, namely $\beta_{2,5}$ and $\beta_{3,6}$. For $\beta_{2,5}$ we are still able to show that almost all graphs $G$ with $\beta_{2,5}(I_G)=0$ satisfy $\reg(I_G)=2$, using a very ad hoc argument involving critical graphs. For $\beta_{3,6}$ we show that this is not possible.

For any Betti number on row\til$3$, Hochster's formula becomes
$$\beta_{i,3+i}(I_G)=\sum_{W\in\binom{V_G}{3+i}}\dim_\mathbb{K} \tilde H_1\big(\mathrm{Ind}(G)[W]\big).$$
The dimension of the degree-one homology of a complex $\Delta$ is the number of one-dimensional holes of $\Delta$. In our case $\Delta=\mathrm{Ind}(G)$ is a flag complex, and such a complex has one-dimensional holes if and only if the $1$-skeleton (that is, the underlying graph of $\Delta$, or equivalently the complement of $G$) contains an induced $k$-cycle with $k\ge4$, namely if and only if the $1$-skeleton is not chordal. In short,  $H_1(\mathrm{Ind}(G)) \ne 0$ if and only if the complement of $G$ is not chordal.

The Betti number $\beta_{2,5}(I_G)$ is determined by the number of induced subgraphs of $G$ on \emph{five} vertices whose complement contains some induced cycles $C_k$ with $k\ge4$. The only graph on five vertices with an induced copy of\til$C_5$ is $C_5$ itself. As for $C_4$, one may start from $C_4$ and add a fifth vertex in all possible ways without destroying the presence of an induced $C_4$, as follows:
$$
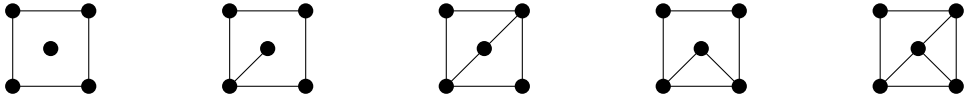
\begin{tikzpicture}[>=latex]
\draw (0,0)--(1,0)--(1,1)--(0,1)--cycle;
\fill (0,0) circle (0.1);
\fill (1,0) circle (0.1);
\fill (1,1) circle (0.1);
\fill (0,1) circle (0.1);
\fill (.5,.5) circle (0.1);
\end{tikzpicture}\qquad\qquad
\begin{tikzpicture}[>=latex]
\draw (0,0)--(1,0)--(1,1)--(0,1)--cycle;
\fill (0,0) circle (0.1);
\fill (1,0) circle (0.1);
\fill (1,1) circle (0.1);
\fill (0,1) circle (0.1);
\fill (.5,.5) circle (0.1);
\draw (0,0)--(.5,.5);
\end{tikzpicture}\qquad\qquad
\begin{tikzpicture}[>=latex]]
\draw (0,0)--(1,0)--(1,1)--(0,1)--cycle;
\fill (0,0) circle (0.1);
\fill (1,0) circle (0.1);
\fill (1,1) circle (0.1);
\fill (0,1) circle (0.1);
\fill (.5,.5) circle (0.1);
\draw (0,0)--(.5,.5);
\draw (1,1)--(.5,.5);
\end{tikzpicture}\qquad\qquad
\begin{tikzpicture}[>=latex]
\draw (0,0)--(1,0)--(1,1)--(0,1)--cycle;
\fill (0,0) circle (0.1);
\fill (1,0) circle (0.1);
\fill (1,1) circle (0.1);
\fill (0,1) circle (0.1);
\fill (.5,.5) circle (0.1);
\draw (0,0)--(.5,.5);
\draw (1,0)--(.5,.5);
\end{tikzpicture}\qquad\qquad
\begin{tikzpicture}[>=latex]
\draw (0,0)--(1,0)--(1,1)--(0,1)--cycle;
\fill (0,0) circle (0.1);
\fill (1,0) circle (0.1);
\fill (1,1) circle (0.1);
\fill (0,1) circle (0.1);
\fill (.5,.5) circle (0.1);
\draw (0,0)--(.5,.5);
\draw (1,1)--(.5,.5);
\draw (1,0)--(.5,.5);
\end{tikzpicture}
$$
The complements of the graphs above are, respectively,
$$
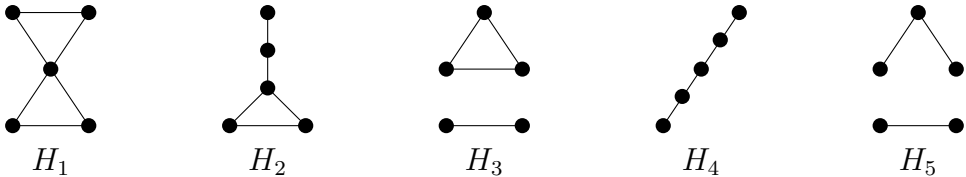
\begin{tikzpicture}[>=latex]
\draw (0,0)--(1,0)--(.5,.75)--cycle;
\fill (0,0) circle (0.1);
\fill (1,0) circle (0.1);
\fill (.5,.75) circle (0.1);
\fill (0,1.5) circle (0.1);
\fill (1,1.5) circle (0.1);
\draw (0,1.5)--(.5,.75)--(1,1.5)--cycle;
\coordinate [label=$H_1$] (G) at (.5,-.8);
\end{tikzpicture}\qquad\qquad
\begin{tikzpicture}[>=latex]
\draw (0,0)--(1,0)--(.5,.5)--cycle;
\fill (0,0) circle (0.1);
\fill (1,0) circle (0.1);
\fill (.5,.5) circle (0.1);
\fill (.5,1) circle (0.1);
\fill (.5,1.5) circle (0.1);
\draw (.5,.5)--(.5,1)--(.5,1.5);
\coordinate [label=$H_2$] (G) at (.5,-.8);
\end{tikzpicture}\qquad\qquad
\begin{tikzpicture}[>=latex]
\draw (0,0)--(1,0);
\fill (0,0) circle (0.1);
\fill (1,0) circle (0.1);
\fill (1,.75) circle (0.1);
\fill (0,.75) circle (0.1);
\fill (.5,1.5) circle (0.1);
\draw (0,.75)--(.5,1.5)--(1,.75)--cycle;
\coordinate [label=$H_3$] (G) at (.5,-.8);
\end{tikzpicture}\qquad\qquad
\begin{tikzpicture}[>=latex]
\draw (0,0)--(1,1.5);
\fill (0,0) circle (0.1);
\fill (1,1.5) circle (0.1);
\fill (.25,.3875) circle (0.1);
\fill (.5,.75) circle (0.1);
\fill (.75,.75+.3875) circle (0.1);
\coordinate [label=$H_4$] (G) at (.5,-.8);
\end{tikzpicture}\qquad\qquad
\begin{tikzpicture}[>=latex]
\draw (0,0)--(1,0);
\fill (0,0) circle (0.1);
\fill (1,0) circle (0.1);
\fill (1,.75) circle (0.1);
\fill (0,.75) circle (0.1);
\fill (.5,1.5) circle (0.1);
\draw (0,.75)--(.5,1.5)--(1,.75);
\coordinate [label=$H_5$] (G) at (.5,-.8);
\end{tikzpicture}
$$
and the complement of $C_5$ is $H_6:=C_5$.

%
%
%
%
%

%
%
%

\begin{lemma}\label{lem:beta25}
The graphs $H_2$, $H_4$ and $H_5$ are critical. Furthermore, by considering the appropriate witnessing pairs, the following hold:
\begin{enumerate}
\item[(a)] almost every $H_2$-free graph is a $(0, 2)$-template or a $(1, 1)$-template,
\item[(b)] almost every $H_4$-free graph is a $(1, 1)$-template or a $(2, 0)$-template,
\item[(c)] almost every $H_5$-free graph is a $(1, 1)$-template or a $(2, 0)$-template.
\end{enumerate}
The graphs $H_1$, $H_3$ and $H_6$ are not critical.
\end{lemma}

\begin{proof}
We just show that $H_5=P_3\sqcup P_2$ is critical, the other cases being very similar.
First of all we determine that the coloring number is $\chi_\mathrm c(H_5)=3$: indeed $(3,0)$, $(2,1)$, $(1,2)$ and $(0,3)$ are covering pairs, whereas $(2,0)$ and $(1,1)$ are the witnessing pairs. To determine the criticality of $H_5$ we need to consider the pairs $(s,t)$ with $s+t=\chi_\mathrm c(H_5)-2=1$.

We start with $(s,t)=(1,0)$. 
After removing a maximal clique, that in this case is an edge, one is either left with $P_3$ or $P_2\sqcup K_1$.
We need to determine what graphs on $n$ vertices, for large $n$, do not contain any induced copies of $P_3$ nor $P_2\sqcup K_1$.
The graph on $n$ vertices without edges is one of them.
If $G$ has at least one edge, let $uv$ be an edge. For any $w\notin\{u,v\}$, also $uw$ and $vw$ need to be edges in order to avoid an induced $P_3$ or $P_2\sqcup K_1$.
Thus $u$ and $v$ are adjacent to all the vertices of $G$. Take any vertex $w\notin\{u,v\}$. We may repeat the same argument for the new edge $uw$, so that $w$  is also adjacent to all vertices of $G$.  And by repeating this for all vertices, it turns out that the graph is complete. So 
$$|\mathcal P(n,\mathcal F(H_5,1,0))|=|\{K_n,\overline{K_n}\}|=2.$$

The other pair to examine is $(s,t)=(0,1)$.
There are two maximal independent sets (up to isomorphism). The first one consists of three vertices: both of the endpoints of $P_3$ and one of $P_2$; and the second one consists of the middle vertex of $P_3$ and a vertex of\til$P_2$. After removing such a maximal independent set, we are left respectively with three isolated vertices or
two isolated vertices. Since we need to care only about the minimal ones by induced inclusion (by Definition\til\ref{defFF}), we just need to consider the graph consisting of two isolated vertices, denoted $\overline{K_2}$.
For large $n$, there is only one graph on $n$ vertices without an induced $\overline{K_2}$, and that is the complete graph $K_n$.  So 
$$|\mathcal P(n,\mathcal F(H_5,0,1))|=|\{K_n\}|=1.$$
Thus $H_5$ is critical.
By Lemma\til\ref{mainthmbb}, almost every $H_5$-graph is an $(s,t)$-template where $(s,t)$ is some witnessing pair of $H_5$, hence we get part (c) of the statement.
\end{proof}

%
%
%

\begin{theorem}
For almost all graphs $G$ with $\beta_{2,5}(I_G)=0$, one has $\reg(I_G)=2$.
\end{theorem}

\begin{proof}
We know that the vanishing of $\beta_{2,5}(I_G)$ implies that $G$ is in particular $\{H_2,H_5\}$-free.
We show that almost all the $\{H_2,H_5\}$-free graphs are $(1,1)$-templates. Then, by Lemma\til\ref{lem:Russ}, almost all graphs with $\beta_{2,5}(I_G)=0$ have edge ideal $I_G$ with $\reg(I_G)=2$.

Recall that the $(0,2)$-templates are the bipartite graphs and the $(1,1)$-templates are the split graphs. 
The number of $(0,2)$-templates, that is, bipartite graphs, is equal to the number of $(2,0)$-templates.

We consider $H_2$ first. Denote 
\begin{align*}
\mathcal T_{(1,1)}(n)&:=\{\text{$(1,1)$-templates on $n$ vertices}\},\\
\mathcal T_{(0,2)}(n)&:=\{\text{$(0,2)$-templates on $n$ vertices}\},\\
\mathcal H(n)&:=\mathcal P(n,H_2)=\{\text{$H_2$-free graphs on $n$ vertices}\}.
\end{align*}
Observe that $\mathcal T_{(1,1)}(n)\cup\mathcal T_{(0,2)}(n)\subset \mathcal H(n)$, and moreover there is only one graph that is both in $\mathcal T_{(1,1)}(n)\cap\mathcal T_{(0,2)}(n)$, namely the star on $n$ vertices. By Corollary\til2.20 of\til\cite{PS92}, the number of $(1,1)$-templates is asymptotically twice the number of $(2,0)$-templates, so that $\lim_{n\to\infty}\frac{|\mathcal T_{(1,1)}(n)|}{|\mathcal T_{(0,2)}(n)|}=2$.
By Lemma\til\ref{lem:beta25}, we know that almost all $H_2$-free graphs are $(0,2)$-templates or $(1,1)$-templates, which is the first equality in
$$1=\lim_{n\to\infty}\frac{|\mathcal H(n)|}{|\mathcal T_{(1,1)}(n)\cup\mathcal T_{(0,2)}(n)|}
=\lim_{n\to\infty}\frac{|\mathcal H(n)|}{\frac32\times|\mathcal T_{(1,1)}(n)|}
=\lim_{n\to\infty}\frac{|\mathcal H(n)|}{3\times|\mathcal T_{(0,2)}(n)|}.$$
This means that, out of the $H_2$-free graphs, asymptotically $2/3$ are $(1,1)$-templates, $1/3$ are $(0,2)$-templates, and $0$ are something else.

By Lemma\til\ref{lem:beta25}, we also know that almost all $H_5$-free graphs are $(1,1)$-templates or $(2,0)$-templates.
And out of the $H_5$-free graphs, asymptotically $2/3$ are $(1,1)$-templates, $1/3$ are $(2,0)$-templates, and $0$ are something else.
The number of $H_2$-free graphs and the number of $H_5$-free graphs are equal, asymptotically.
Thus, aymptotically the intersection of the $H_2$-free graphs and the $H_5$-free graphs consists completely of $(1,1)$-templates.
\end{proof}

\begin{example}
There are $(2,0)$-templates without four-cycles between the two covering
cliques and which still have quite many edges. The canonical extremal examples are incidence graphs of finite projective planes. The first example is the complement of the Heawood graph. (The corresponding simplicial complex is the Bruhat--Tits building of $SL_3(\mathbb Z_2)$.) The Betti table is
$$\begin{array}{c|cccccccccccc}
&0 & 1&   2    &3    &4    &5    &6    &7    &8    &9  &10 &11 \\
\hline
2& 70& 476 &1617 &3388 &4648 &4184 &2394  &826  &161  &14  &-  &-\\
3&  -&   -&    -&   28  &224  &777 &1442 &1547  &994 &385 &84  &\,\,8.
\end{array}$$
One can see that $\beta_{2,5}(I_G)=0$ but the resolution is not $2$-linear. The number $\beta_{3,6}(I_G)=28$ in row $3$ is the number of six-cycles in the Heawood graph.
\end{example}

\begin{remark}
To conclude this section we observe that the same type of argument as for $\beta_{2,5}$, on graphs with five vertices,  cannot be applied to $\beta_{3,6}$, that is, with induced subgraphs on six vertices. This is demonstrated by the existence of critical graphs with completely different templates.
These are some of the critical graphs on six vertices with non-zero homology of degree\til$1$:
\begin{itemize}
\item The disjoint union of two paths on three vertices each, call it $H_1$. Almost all $H_1$-free graphs are $(3,0)$-templates. 
\item The disjoint union of a square and an edge, call it $H_2$. Almost all $H_2$-free graphs are $(2,1)$-templates. 
\item The disjoint union of two triangles, call it $H_3$. Almost all $H_3$-free graphs are $(1,2)$-templates. 
\item The disjoint union of a $K_4$ and an edge, call it $H_4$. Almost all $H_4$-free graphs are $(0,3)$-templates. 
\end{itemize}
One has in fact all possible $(s,t)$-templates with $s+t=3$, and clearly they are not compatible in an argument similar to the one for $\beta_{2,5}$.
\end{remark}


\section{Large homogenous sets: the Erd\H{o}s--Hajnal conjecture}\label{sec:erdos}

A subset of vertices in a graph is called a \emph{homogenous set} if it spans a clique or an independent set. In this section we relate our results to the following famous conjecture:

\begin{conjecture}[Erd\H{o}s--Hajnal \cite{ErHa1,ErHa2}]
For every graph $H$, there is a constant $\tau$ such that any $H$-free graph $G$ has a homogenous set of order at least $|G|^\tau$.
\end{conjecture}
It is a difficult conjecture, and versions where several graphs $H$ are forbidden, or almost all $H$-free graphs $G$ are considered, have been studied before. Sometimes, under stronger assumptions, a stronger result than Erd\H{o}s--Hajnal is true, stating that there is a homogenous set of linear order $\alpha|G|$ for some $\alpha >0$. In an $(s,t)$-template $G$ there is always a linear-order homogenous set with $\alpha=1/(s+t)$ as the vertices of $G$ may be partitioned into $s$ cliques and $t$ independent sets. The following  is an immediate consequence of Theorem~\ref{thm:main1}.

\begin{corollary}
Let $\beta_{i,j}$ be a parabolic Betti number on the $r$-th row of the Betti table, for $r\ge3$. In almost every graph $G$ with $\beta_{i,j}(I_G)=0$ there is a homogenous set of order $|G|/(r-1)$.
\end{corollary}

One may wish to consider any Betti number, not just parabolic ones, and to get rid of the ``almost''. First we recall some results on the  Erd\H{o}s--Hajnal conjecture. We say that a graph  $H$ is a \emph{Erd\H{o}s--Hajnal graph} if it satisfies the conjecture. The \emph{substitution} of a vertex $v$ in a graph $F$ by a graph $H$ is constructed by replacing $v$ by $H$ and making all vertices of $H$ incident to the neighbours of $v$ in $F$. The state of the art is:

\begin{theorem}\label{thm:eh}
A graph $G$ is Erd\H{o}s--Hajnal  if $G$ is 
\begin{itemize}
\item a vertex, an edge or a path on four vertices,
\item a bull (which is a triangle with two horns) \cite{CS08},
\item a five-cycle \cite{CSSS21},
\item the complement of an Erd\H{o}s--Hajnal graph (by definition of homogeneous set), or
\item constructed by substitution from Erd\H{o}s--Hajnal graphs \cite{APS01}.
\end{itemize}
\end{theorem}

We need the following result, and it is rather possible that it might be proved without the substitution technology.

\begin{lemma}\label{lem:eh}
Any cluster (i.e., a disjoint union of cliques) is Erd\H{o}s--Hajnal.
\end{lemma}

\begin{proof}
The complete graphs are Erd\H{o}s--Hajnal because $K_1$ and $K_2$ are, and one gets $K_n$ from $K_{n-1}$ by substituting a vertex by $K_2$. The disjoint union of vertices is Erd\H{o}s--Hajnal, because its complement is a complete graph. The disjoint union of $k$ complete graphs is Erd\H{o}s--Hajnal, as one may start off with the disjoint union of $k$ vertices and then perform $k$ substitutions of them by complete graphs of appropriate orders. 
\end{proof}

In the following result we consider the Betti numbers whose non-vanishing is ``attainable'' by edge ideals: as recalled in Section\til\ref{sec:backalgebra}, they are the numbers located below row\til$2$ of the Betti table and to the right of the main diagonal consisting of the numbers $\beta_{i,\,2i+2}$, including row\til$2$ and that diagonal. This condition is given explicitly in terms of $i$ and $j$: for any $i\ge0$, we are interested in the Betti numbers $\beta_{i,j}$ with $2\le j\le 2i+2$.

\begin{proposition}
Let\/ $i\ge0$ and\/ $2\le j\le 2i+2$.
Then there exists $\tau>0$ such that
if $\beta_{i,j}(I_G)=0$ then there is a homogenous set of order $|G|^\tau$ in $G$.
\end{proposition}

\begin{proof}
Use Lemma~\ref{homologyoffatmatching} to find a disjoint union of complete graphs $H$ such that $\beta_{i,j}(I_H)>0$. The graph $H$ is Erd\H{o}s--Hajnal according to Lemma~\ref{lem:eh}.
\end{proof}

In general it is unclear to what extent this might be improved. By Theorem\til\ref{thm:frobg}, the graphs with only positive Betti numbers on the second line are exactly the complements of  chordal graphs. By Theorem\til\ref{thm:almchosplit}, almost all chordal graphs are split, and thus have a homogenous set spanning half the graph. But the disjoint union of $\sqrt n$ cliques on $\sqrt n$ vertices each is a chordal graph whose complement only has sub-linear homogenous sets.

\section{Concluding remarks}\label{sec:lastsec}

\subsection{Towards a moduli space of graphs}\label{sec:spaceofgraphs}

The main structural property for graphs that comes up in this paper is that of $(s,t)$-template. In this section we investigate the connectedness of the set of $(s,t)$-templates inside  the ``space of graphs'' introduced below. This is not directly related to the main results in this paper, but it may be useful for future developments.

Consider the following labeled graph $\mathcal{G}_n$: the vertices are all the unlabeled graphs on $n$ vertices, and two such graphs $G_1$ and $G_2$ are connected by an edge in $\mathcal{G}_n$ if one can obtain $G_1$ by adding an edge to $G_2$, or vice versa. 
We note that one
 may partition the vertices of $\mathcal{G}_n$ in $\binom n2+1$ independent sets, based on the cardinality of the edge set: the $i$-th independent set $I_i$ consists of all unlabeled graphs on $n$ vertices with exactly $i$ edges. Notice that the only edges of $\mathcal{G}_n$ are between independent sets of adjacent cardinalities (that is, between $I_i$ and $I_{i+1}$).  In particular the graph $\mathcal{G}_n$ is bipartite: one can partition the vertices of $\mathcal{G}_n$ in two independent sets consisting of the graphs with an odd and even number of edges.
  Clearly, taking the complement is an automorphism $G\mapsto\overline G$ of $\mathcal{G}_n$.

\begin{lemma}\label{lem:connected}
Let $G$ be an $(s,t)$-template with $s\ge1$ and $e<\binom n2$ edges. Then there is a way of adding one edge  to $G$ so that the resulting graph is still an $(s,t)$-template.
\end{lemma}

\begin{proof}
By definition, the graph $G$ can be covered with $s$ cliques and $t$ independent sets, that is, $s+t$ homogeneous sets. If there exist two non-adjacent vertices from distinct homogeneous subsets, we may add an edge to connect them, thus obtaining a new $(s,t)$-template with one more edge than $G$.
 If not, it means that there is at least one independent set~$S$ that has more than one vertex, because $e<\binom n2$. We may then consider a new covering of the graph $G$ where a vertex $v$ of~$S$ is ``moved'' it to any of the $s\ge1$ cliques. Then this vertex $v$ may be connected to a vertex of $S$, again getting an $(s,t)$-template with one extra edge.
\end{proof}

\begin{proposition} 
For any non-negative integers $s$ and $t$, the set of $(s,t)$-templates is connected in $\mathcal{G}_n$. 
\end{proposition}

\begin{proof}
We first prove the claim for $s\ge1$. Notice that, for any $s\ge1$ and for any $t\ge0$, the complete graph $K_n$ is an $(s,t)$-template.
Let $s\ge1$ and pick any two $(s,t)$-templates $G$ and $H$. By Lemma~\ref{lem:connected}, there is a path from $G$ to the complete graph $K_n$, and there also is a path from $H$ to~$K_n$. Hence, $G$ and $H$ are connected.

For the case with $s=0$ and $t\ge1$, just  take the complements of the elements of $\mathcal{G}_n$, which is an automorphism of $\mathcal{G}_n$ as observed 
above and notice moreover that a graph $G$ is an $(s,t)$-template if and only if $\overline G$ is a $(t,s)$-template.

The trivial case of $(0,0)$-templates is the empty set, which is connected.
\end{proof}

\subsection{Future directions}\label{sec:future}

This paper is part of a research project whose aim is much more general than the specific theoretical tools employed in this text: we hope that some of the results we proved for edge ideals of graphs can be generalized to more general monomial ideals.
With squarefree monomial ideals in mind, we did not focus initially on employing the critical graphs of Balogh and Butterfield\til\cite{BaBu} but the new and powerful theory of hypergraph containers by Balogh, Morris and Samotij\til\cite{BMS15}, and Saxton and Thomason\til\cite{ST15}. Although some readers might consider the technicalities in employing critical graphs in our setting slightly daunting, it is still several levels less technical than using the complete and general theory of hypergraph containers. Establishing the desirable properties of parabolic clusters would be less direct and require more development without the explicit characterization of critical graphs. That said, hypergraph container theory would allow for a much more general setting than in this paper, and several interesting future directions emanate from that setup.

An advantage of hypergraph containers is that one may understand not only the typical behavior of graphs without certain substructures, but also typical behaviors of those with very few undesirable substructures. A future direction would be to translate that into the behavior of Betti numbers along the following lines:
Let $G$ be some model for a random graph on $n$ vertices, for example an instance of the generalized Erd\H{o}s--R\'enyi models called \emph{graph limits}\til\cite{LL12}.
Then the Betti number $\beta_{i,j}(I_G)$ is of the order $n^j$. In this text we focused on studying almost all graphs with $\beta_{i,j}(I_G)=0$, but with hypergraph containers one should get results for all graphs with $\beta_{i,j}(I_G)<n^{j-1}$. We expect something as follows:

\begin{conjecture}
For every non-negative real number $C$, there exists a non-negative real number $\alpha(C)$ satisfying the following. Let $\beta_{i,j}$ be a parabolic Betti number on the $r$-th row of the Betti table, for some $r \geq 3$. Then, for almost every graph $G$ on $n$ vertices with $\beta_{i,j}(I_G) \leq Cn^{j-1}$, one may construct a new graph $H$ from $G$ by changing at most $\alpha(C)n$ edges so that $reg(I_H)=r-1.$
\end{conjecture}

The case $C=0$ with $\alpha(C)=0$ is the first part of the main result of this paper.

Using a model of random graphs other than the uniform or Erd\H{o}s--R\'enyi should also be attainable within the context of hypergraph containers. For example one may consider the ordinary graph limits (or graphons) introduced by Lovasz et al.\til\cite{LL12} or monomial ideals from more general random combinatorial structures than graphs or hypergraphs as described in Chapter 7 of Kallenberg's textbook\til\cite{K06}.  

The main question driving us was: What does it mean that a Betti number vanishes? It is  natural to generalize this question to the vanishing of several Betti numbers. As an anonymous referee very kindly suggested, it would be possible to apply these results to study (at least asymptotically) the conjecture below (see\til\cite{ACSsubadd,McSubadd}). For an ideal $I$, let $t_i=t_i(I):= \mathrm{max}\{j \mid\beta_{i,j} (I)\ne0\}$. 
\begin{conjecture}[sub-additivity for edge ideals]
For an edge ideal $I_G$ one has $t_a+t_b\ge t_{a+b}$ for all $a$ and $b$ where $a+b$ is at most the projective dimension of $I_G$.
\end{conjecture}
In the conjecture above, in particular one has $\beta_{a,\,t_a+1}=0$ and $\beta_{b,\,t_b+1}=0$, by definition of $t_a$ and $t_b$. So it is natural to ask, for instance in the case that these Betti numbers are parabolic, whether the number $\beta_{a+b,\,j}$ is somehow controlled, at least asymptotically for almost all graphs.

Lastly, how about not necessarily monomial ideals? The most technical results from graph theory used in this paper have roots in continuous probability theory, and in some sense they work better in the classical analytic or measure-theoretic setting. It is completely conceivable that one could put interesting measures on ideals $I$ in a ring $R$ and get similar results as in our main theorem:  If a parabolic Betti number $\beta_{i,j}(I)$ on row $r$ in the Betti table vanishes, then $\reg(I)=r-1$ with probability\til$1$. In some sense the graph theory is abstractly awkward in the setup of this paper for both the commutative algebra related to algebraic geometry and for the probability theory hiding in the shadows behind the critical graphs, but it provides an explicitness that makes it possible to prove interesting theorems to begin with. With those theorems now proved, more general future directions have opened up.

\bibliographystyle{amsplain}
\bibliography{references-extremal}

\end{document}